\documentclass[11pt]{amsart}

\usepackage{graphicx}
\usepackage[USenglish,english]{babel}
\usepackage[T1]{fontenc}
\usepackage[latin1]{inputenc}
\usepackage{amsfonts}
\usepackage{amssymb}
\usepackage{amsthm}
\usepackage{amsmath}
\usepackage{tikz}
\usetikzlibrary{decorations.pathreplacing}
\usepackage{float}
\usepackage{enumerate}
\usepackage{accents}
\usepackage{mathtools}
\usepackage{mathrsfs}
\usepackage{comment}
\usepackage{afterpage}
\usepackage{bbm}
\usepackage{dsfont}
\usepackage{stmaryrd}
\usepackage{hyperref}
\usepackage{mleftright}
\usepackage{subfig}
\usepackage{faktor}

\theoremstyle{plain}
\newtheorem{thm}{Theorem}[section]
\newtheorem{lem}[thm]{Lemma}
\newtheorem{prop}[thm]{Proposition}
\newtheorem{cor}[thm]{Corollary}
\newtheorem*{thm*}{Main Theorem}
\newtheorem*{prop*}{Proposition}
\newtheorem*{cor*}{Corollary}
\newtheorem{thmintro}{Theorem}

\theoremstyle{definition}
\newtheorem{mydef}[thm]{Definition}
\newtheorem{exam}[thm]{Example}
\newtheorem{rem}[thm]{Remark}
\newtheorem*{quest*}{Question}

\DeclareMathOperator{\Ima}{Im}

\title{M\"obius structures, quasi-metrics and completeness}

\author{Merlin Incerti-Medici}
\address{Institut f\"ur Mathematik, Universit\"at Z\"urich, Switzerland}
\email{merlin.incerti-medici@math.uzh.ch}

\begin{document}

\maketitle

\begin{abstract}
We study cross ratios from an axiomatic viewpoint, also known as the study of M\"obius spaces. We characterise cross ratios induced by quasi-metrics in terms of topological properties of their image. Furthermore, we generalise the notions of Cauchy-sequences and completeness to M\"obius spaces and prove the existence of a unique completion under an extra assumption that, again, can be expressed in terms of the image of the cross ratio.
\end{abstract}

\tableofcontents




\section{Introduction} \label{sec:Introduction}

Let $Z$ be a set, $\rho$ a metric on $Z$, possibly with a point at infinity (see section \ref{sec:Preliminaries} for definitions). We can define the cross ratio induced by $\rho$ with the formula
\begin{equation} \label{eq:CrossRatio}
cr(z_1, z_2, z_3, z_4) := \frac{\rho(z_1, z_2) \rho(z_3, z_4)}{\rho(z_1, z_3) \rho(z_2, z_4)},
\end{equation}
where the quotient of any two infinite distances equals $1$, i.e.\,infinite distances cancel each other. Provided that no three points in the quadruple $(z_1, \dots, z_4)$ coincide, this yields a well-defined number in $[0, \infty]$.

Cross ratios arise naturally in the study of negatively curved spaces: If $X$ is a $\mathrm{CAT(-1)}$ space, we can define its boundary at infinity, which can be endowed with a family of metrics $\{ \rho_x \}_{x \in X}$, called visual metrics. It is a classical result by Bourdon that, for a $\mathrm{CAT(-1)}$ space, all visual metrics induce the same cross ratio on the boundary. Therefore, the cross ratio provides us with an intrinsic geometric structure on the boundary at infinity. This allows us to think about the pair $(\partial X, cr)$ as a topological space with a geometric structure of its own, which leads to the study of cross ratios from an axiomatic viewpoint (see for example \cite{Hamenstadt97, Buyalo}). In this context, cross ratios are also referred to as M\"obius structures and a set equipped with a M\"obius structure will be called a M\"obius space.

In \cite{Buyalo}, Buyalo showed how M\"obius structures give rise to a topology, called M\"obius topology. Furthermore, he showed that every M\"obius structure is induced by a semi-metric, i.e.\,every M\"obius structure arises from formula (\ref{eq:CrossRatio}) if $\rho$ is a semi-metric, that is, it satisfies the same properties as a metric, except for the triangle inequality. Between semi-metrics and metrics there is the notion of a $K$-quasi-metric, which satisfies a weak triangle inequality (see section \ref{sec:Preliminaries} for precise definitions). Quasi-metrics are of particular interest in the study of cross ratios because of involutions. Given a metric $\rho$, its involution at a point $o \in Z$ is defined by
\[ \rho_{o}(z, z') = \frac{Ê\rho(z,z') }{ \rho(z, o) \rho(o, z') }. \]
A direct computation shows that $\rho_o$ induces the same cross ratio as $\rho$. However, if $\rho$ is a metric, the map $\rho_o$ may no longer be a metric which leads to technical complications when studying cross ratios purely from a metric point of view. Quasi-metrics have the advantage that, given a quasi-metric $\rho$, the involution $\rho_o$ is again a quasi-metric (cf. Proposition 5.3.6 in \cite{BS}). Quasi-metrics are weaker than metrics in many ways. For example, they do not enjoy the same continuity properties as metrics, as we will see in example \ref{exam:Continuity}. However, M\"obius structures induced by quasi-metrics have several nice topological features, which, together with the observation on involutions above, motivates their study.\\

In this paper, we provide a characterisation of those M\"obius structures that are induced by quasi-metrics in terms of the image of the cross ratio. We then study the M\"obius topology introduced by Buyalo and show that, if the cross ratio is induced by a metric, the metric topology and the M\"obius topology coincide. Finally, if a M\"obius structure is induced by a quasi-metric that satisfies an additional symmetry condition, we can define the notion of Cauchy-sequences for such a M\"obius structure. The main results of this paper are the following:

\begin{thmintro} \label{thmintro:MobiusvsMetricTopology}
Let $(Z, \rho)$ be a metric space, $M$ the M\"obius structure induced by $\rho$. Denote the metric topology induced by $\rho$ by $\mathcal{T}_{\rho}$ and the M\"obius topology induced by $M$ by $\mathcal{T}_{M}$. Then $\mathcal{T}_{\rho} = \mathcal{T}_{M}$.
\end{thmintro}

\begin{thmintro} \label{thmintro:MobiusvsMetricCompleteness}
Let $(Z, \rho)$ be a (possibly extended) metric space and denote the induced M\"obius structure by $M$. Equivalent are:
\begin{enumerate}
\item[1)] $(Z, M)$ is complete as a M\"obius space.
\item[2)] $(Z, \rho)$ is complete as a metric space and is either bounded or has a point at infinity.
\end{enumerate}
\end{thmintro}

\begin{thmintro} \label{thmintro:Completion}
Let $(Z, M)$ be a M\"obius space that satisfies the (symmetry)-condition. Then there exists a complete M\"obius space $(\overline{Z}, \overline{M})$ with a M\"obius embedding $i_Z : Z \hookrightarrow \overline{Z}$ such that $i_Z(Z)$ is dense in $\overline{Z}$.

Furthermore, if $(Z', M')$ is a complete M\"obius space with a M\"obius embedding $i : Z \hookrightarrow Z'$ such that $i(Z)$ is dense in $Z'$, then there exists a unique M\"obius equivalence $f : \overline{Z} \rightarrow Z'$ such that $i = f \circ i_Z$.
\end{thmintro}

The rest of the paper is organised as follows. In section \ref{sec:Preliminaries}, we give precise definitions for the terminology we will require. In section \ref{sec:MobiusStructuresQuasiMetrics}, we show the characterisation of M\"obius structures induced by quasi-metrics. In section \ref{sec:MobiusTopology}, we review Buyalos definition of the M\"obius topology and prove Theorem \ref{thmintro:MobiusvsMetricTopology}. In section \ref{sec:CauchySequences}, we introduce Cauchy-sequences and prove Theorem \ref{thmintro:MobiusvsMetricCompleteness}. In section \ref{sec:Completion}, we construct the completion and prove Theorem \ref{thmintro:Completion}\\
\\
{\bf Acknowledgments} The author is grateful to Viktor Schroeder for many discussions and helpful advice.




\section{Preliminaries} \label{sec:Preliminaries}

Let $Z$ be a set, $\rho : Z \times Z \rightarrow \mathbb{R}$ a map. We say that $\rho$ is a {\it semi-metric} if it is symmetric, non-negative and $\rho(z,z') = 0$ if and only if $z = z'$. We say that $\rho$ is a {\it $K$-quasi-metric}, where $K \geq 1$, if it is a semi-metric and for all $x, y, z \in Z$, we have $\rho(x,z) \leq K \max(\rho(x,y), \rho(y,z))$. Finally, we say $\rho$ is a {\it metric} if it is a semi-metric and for all $x, y, z \in Z$, we have $\rho(x,z) \leq \rho(x,y) + \rho(y,z)$. Generalizing the definition of a metric, we say that $\rho : Z \times Z \rightarrow [0, \infty]$ is an {\it extended metric} if there exists exactly one point $\omega \in Z$, such that for all $x \in Z \diagdown \{ \omega \}$, $\rho(x, \omega) = \infty$, $\rho( \omega, \omega) = 0$ and the restriction of $\rho$ to $Z \diagdown \{ \omega \} \times Z \diagdown \{ \omega \}$ is a metric. We call $\omega$ the {\it point at infinity} with respect to $\rho$. A motivating example for this notion is the Riemannian sphere, seen as the union $\mathbb{C} \cup \{ \infty \}$. We define the notions of extended semi-metrics and extended $K$-quasi-metrics analogously.

We call an $n$-tuple $(z_1, \dots, z_n) \in Z^n$ {\it non-degenerate} if and only if for all $i \neq j$, we have $z_i \neq z_j$.

Given a semi-metric $\rho$, we can define a cross ratio. The cross ratio will be defined on admissible quadruples.

\begin{mydef}
A quadruple $(z_1, z_2, z_3, z_4) \in Z^4$ is admissible if there exists no triple $i \neq j \neq k \neq i$ such that $z_i = z_j = z_k$. We denote the set of admissible quadruples by $\mathcal{A}$.
\end{mydef}

We define the cross ratio induced by $\rho$ as follows:
\[ \forall (z_1, z_2, z_3, z_4) \in \mathcal{A} : cr(z_1, z_2, z_3, z_4) := \frac{ \rho(z_1, z_2) \rho(z_3, z_4) }{ \rho(z_1, z_3) \rho(z_2, z_4) } \in [0, \infty]. \]

Admissible quadruples are exactly those quadruples, for which the expression above does not yields division of zero by zero for any permutation of the points $z_i$.

We also define the cross ratio triple. Denote
\[Ê\Delta := \{ (a : b : c) \in \mathbb{R}P^2 \vert a, b, c > 0 \}, \]
\[ \overline{\Delta} := \Delta \cup \{ (0: 1 : 1), (1: 0: 1), (1 : 1 : 0) \}.\]
The cross ratio triple induced by $\rho$ is a map $crt : \mathcal{A} \rightarrow \overline{\Delta}$ defined by
\[ crt(z_1, z_2, z_3, z_4) := ( \rho(z_1, z_2) \rho(z_3, z_4) : \rho(z_1, z_3) \rho(z_2, z_4) : \rho(z_1, z_4) \rho(z_2, z_3) ). \]

Admissible quadruples are exactly those quadruples, for which at most one entry of the cross ratio triple is zero.\\

We can generalise these definitions to extended semi-metrics by using the following convention. Let $\omega \in Z$ be the point at infinity with respect to $\rho$. Fractions of the form $\frac{ \rho(\omega, z) }{ \rho(\omega, z') }$ for $z, z' \in Z \diagdown \{ \omega \}$ can be replaced by $1$, based on the principle that `infinite distances cancel each other'. In other words, if $z_1, z_2, z_3 \in Z \diagdown \{ \omega \}$, then
\begin{equation*}
\begin{split}& cr(z_1, z_2, z_3, \omega) = \frac{ \rho(z_1, z_2) }{ \rho(z_1, z_3) },\\
& cr(z_1, z_2, \omega, \omega) = 0,\\
& cr(z_1, \omega, \omega, z_2) = 1,\\
& crt(z_1, z_2, z_3, \omega) = ( \rho(z_1, z_2) : \rho(z_1, z_3) : \rho(z_2, z_3) ),\\
& crt(z_1, z_2, \omega, \omega) = ( 0 : 1 : 1 ).
\end{split}
\end{equation*}

It turns out that the maps $cr$ and $crt$ determine each other. If $crt(z_1, z_2, z_3, z_4) = (a : b : c)$, then $cr(z_1, z_2, z_3, z_4) = \frac{a}{b}$. On the other hand, if we denote
\[ cr(z_1, z_3, z_4, z_2) := \alpha,Ê\]
\[ cr(z_1, z_4, z_2, z_3) := \beta,Ê\]
\[ cr(z_1, z_2, z_3, z_4) := \gamma,Ê\]
then
\[ crt(z_1, z_2, z_3, z_4) = \left( \gamma^{\frac{1}{3}} \beta^{-\frac{1}{3}} : \alpha^{\frac{1}{3}} \gamma^{-\frac{1}{3}} : \beta^{\frac{1}{3}} \alpha^{-\frac{1}{3}} \right).Ê\]

In order to study the properties of the cross ratio, it is sometimes useful to reformulate the cross ratio in an additive manner. Denote
\[Ê\overline{L_4} := \{ (x,y,z) \in \mathbb{R}^3 \vert x + y + z = 0 \} \cup \{ (0, \infty, -\infty), (-\infty, 0, \infty), (\infty, - \infty, 0) \}. \]
We define the cross difference $M : \mathcal{A} \rightarrow \overline{L_4}$ induced by $\rho$ to be
\[ M(z_1, z_2, z_3, z_4) := \left( \ln \left( cr(z_1, z_3, z_4, z_2) \right), \ln \left( cr(z_1, z_4, z_2, z_3) \right), \ln\left( cr(z_1, z_2, z_3, z_4) \right) \right).Ê\]
The maps $M$ and $cr$ determine each other.\\

We end this section with a construction that allows us to construct different semi-metrics that induce the same cross ratio. Let $\rho$ be an extended semi-metric and let $o \in Z$ be a point such that for all $z \neq o$, $\rho(z,o) > 0$. We define the {\it involution of $\rho$ at $o$} by
\[ \rho_o(x,y) := \frac{ \rho(x,y)Ê}{ \rho(x,o) \rho(o,y) }. \]
Note that $o$ lies at infinity with respect to $\rho_o$ and, if $\omega$ is a point at infinity with respect to $\rho$, then
\[ \rho_o(x,\omega) = \frac{1}{\rho(x,o)}.Ê\]

Note that, if $\rho$ was an extended semi-metric, then $\rho_o$ is again an extended semi-metric. In \cite{BS}, Proposition 5.3.6, Buyalo-Schroeder prove that for any extended $K$-quasi-metric $\rho$, its involution $\rho_o$ is a $K'^2$-quasi-metric for some $K' \geq K$. A direct computation shows that $\rho$ and $\rho_o$ induce the same cross ratio.


\section{M\"obius structures and quasi-metrics} \label{sec:MobiusStructuresQuasiMetrics}

Consider the ordered triple $((12)(34), (13)(42), (14)(23))$. The symmetric group of four elements $\mathcal{S}_4$ acts on this triple by permuting the numbers 1-4. Whenever $\sigma \in \mathcal{S}_4$ acts on the numbers, it induces a permutation on the triple. Define $\varphi(\sigma) \in \mathcal{S}_3$ to be the permutation on the triple induced by the action of $\sigma$. It is easy to check that $\varphi : \mathcal{S}_4 \rightarrow \mathcal{S}_3$ is a group homomorphism. One can interpret the expression $(12)(34)$ to denote two opposite edges of a tetrahedron whose corners are labeled by the numbers 1-4. In this interpretation, $\varphi$ is the group homomorphism that sends a permutation of the corners to the induced permutation of pairs of opposite edges.\\

Let $Z$ be a set with at least three points. For any semi-metric, denote its set of admissible quadruples by $\mathcal{A}$ (recall that all semi-metrics have the same admissible quadruples). We can now define a cross ratio axiomatically.

\begin{mydef} \label{def:GeneralizedMobiusStructure}

Let $Z$ be a set with at least three points. A map $M : \mathcal{A} \rightarrow \overline{L_4}$ is called a {\it  M\"obius structure} if and only if it satisfies the following conditions:

\begin{enumerate}

\item[1)] For all $P \in \mathcal{A}$ and all $\pi \in \mathcal{S}_4$, we have

\[ M(\pi P) = sgn(\pi)\varphi(\pi)M(P). \]

\item[2)] For $P \in \mathcal{A}$, $M(P) \in L_4$ if and only if $P$ is non-degenerate.

\item[3)] For $P = (x,x,y,z)$, we have $M(P) = (0, \infty, -\infty)$.

\item[4)] Let $(x,y, \omega, \alpha, \beta)$ be an admissible $5$-tuple $(x, y, \omega, \alpha, \beta)$ such that $(\omega, \alpha, \beta)$ is a non-degenerate triple, $\alpha \neq x \neq \beta$ and $\alpha \neq y \neq \beta$. Then, there exists some $\lambda = \lambda(x,y,\omega,\alpha,\beta) \in \mathbb{R} \cup \{ \pm \infty \}$ such that

\[ M(\alpha x \omega \beta) + M(\alpha \omega y \beta) - M(\alpha x y \beta) = (\lambda, -\lambda, 0). \]

Moreover, when $(\omega, \alpha, \beta)$ is non-degenerate, $x \neq \beta$ and $y \neq \alpha$, the first component of the left-hand-side expression is well-defined. Analogously, the second component of the left-hand-side expression is well-defined when $(\omega, \alpha, \beta)$ is non-degenerate, $x \neq \alpha$ and $y \neq \beta$.

\end{enumerate}

The pair $(Z,M)$ is called a {\it  M\"obius space}.

\end{mydef}

Given $M$, we obtain a map $cr : \mathcal{A} \rightarrow [0, \infty]$ and a map $crt : \mathcal{A} \rightarrow \overline{\Delta}$ using the formulas from section \ref{sec:Preliminaries}. Abusing notation, we will also refer to $(Z, cr)$ and $(Z, crt)$ as  M\"obius spaces.

It is a straightforward computation to show that for any semi-metric $\rho$, the induced cross difference $M$ is a  M\"obius structure. Buyalo proved in \cite{Buyalo} that the converse is true as well: Every  M\"obius structure is the cross difference of a semi-metric. We also have a characterisation of M\"obius structures that are induced by quasi-metrics.

\begin{mydef} \label{def:MobiusStructure}

Let $Z$ be a set with at least three points. A map $M : \mathcal{A} \rightarrow \overline{L_4}$ is called a {\it strong M\"obius structure} if it is a  M\"obius structure and the induced map $crt$ satisfies the following condition:

\begin{enumerate}

\item [(corner)] There exist open neighbourhoods of $(1:0:0), (0:1:0), (0:0:1)$, such that the image of $crt$ doesn't intersect these neighbourhoods.

\end{enumerate}

\end{mydef}

The remainder of this section is devoted to proving the following result.

\begin{prop} \label{prop:MobiusStructuresQuasiMetrics}
Let $(Z, M)$ be a  M\"obius structure. There exists an extended quasi-metric $\rho$ inducing $M$ if and only if $M$ is a strong M\"obius structure.

Furthermore, whenever there exists an extended $K$-quasi-metric inducing $M$, there exists a bounded $K^2$-quasi-metric inducing $M$.
\end{prop}

\begin{figure}
\begin{tikzpicture}
\draw [dashed] (6,0) -- (9,0);
\draw [dashed] (6,0) -- (7.5, 2.6);
\draw [dashed] (9,0) -- (7.5, 2.6);
\draw [fill] (6,0) circle [radius=0.05];
\draw [fill] (9,0) circle [radius=0.05];
\draw [fill] (7.5,2.6) circle [radius=0.05];
\draw [fill] (6.75,1.3) circle [radius=0.03];
\draw [fill] (8.25,1.3) circle [radius=0.03];
\draw [fill] (7.5,0) circle [radius=0.03];
\draw (6,0) circle [radius=0.5];
\draw (9,0) circle [radius=0.5];
\draw (7.5,2.6) circle [radius=0.5];
\node [below left] at (5.7,-0.3) {$(1:0:0)$};
\node [below right] at (9.3,-0.3) {$(0:1:0)$};
\node [above] at (7.5,3) {$(0:0:1)$};
\node at (7.5,1.2) {$\overline{\Delta}$};
\end{tikzpicture}
\caption{A  M\"obius structure $crt$ satisfies the (corner)-condition if and only if we can find open neighbourhoods as depicted above, such that the image of $crt$ in $\overline{\Delta}$ doesn't intersect these neighbourhoods.}
\end{figure}
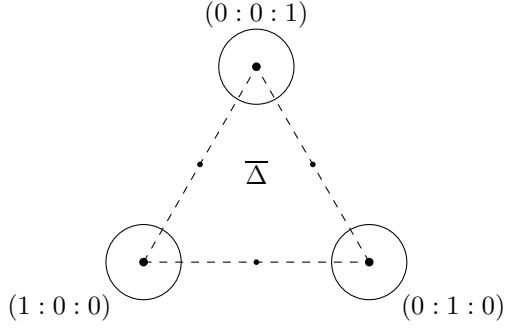

We begin by proving that quasi-metrics induce strong M\"obius structures.

\begin{lem} \label{lem:cornercondition}
Let $Z$ be a set, $\rho$ a quasi-metric on $Z$ and $crt$ the cross ratio induced by $\rho$. Then $crt$ satisfies the (corner)-condition and, therefore, the induced cross difference $M$ is a strong M\"obius structure
\end{lem}

\begin{proof}
Let $\rho$ be a $K$-quasi-metric on $Z$, $M$ the induced  M\"obius structure and $crt$ the induced cross ratio triple. Let $(w, x, y, z)$ be an admissible quadruple. We want to show that $crt(w, x, y, z)$ cannot be close to any of the three corner points. We will show this for the corner point $(0:0:1)$. The others work analogously.

In order for the point $crt(w, x, y, z)$ to be close to $(0:0:1)$, the ratio between the first and third component has to be small, as does the ratio between the second and the third component. We will show that this cannot happen. To prove this, we need to make several case distinctions. We leave it to the reader to check that all cases can be handled analogously by simply permuting the roles and properties of $w,x,y,z$.

Let $\epsilon > 0$. Consider $crt(w, x, y, z) = ( \rho(w,x)\rho(y,z) : \rho(w,y)\rho(x,z) : \rho(w,z)\rho(x,y))$ and suppose
\begin{equation*}
\begin{split}
\max( \rho(w,x)\rho(y,z), \rho(w,y)\rho(x,z) ) = \epsilon.
\end{split}
\end{equation*}

We want to bound $\rho(w,z)\rho(x,y)$ in terms of $\epsilon$, proving that the ratios $\frac{\rho(w,x)\rho(y,z)}{\rho(w,z)\rho(x,y)}, \frac{\rho(w,y)\rho(x,z)}{\rho(w,z)\rho(x,y)}$ cannot become too small. Assume without loss of generality that
\begin{equation*}
\begin{split}
&\rho(w,x) \leq \rho(y,z),\\
&\rho(w,y) \leq \rho(x,z)
\end{split}
\end{equation*}
and thus
\begin{equation*}
\begin{split}
& \rho(w,x) \leq \sqrt{\epsilon},\\
& \rho(w,y) \leq \sqrt{\epsilon}.
\end{split}
\end{equation*}

Since $\rho$ is a $K$-quasi-metric, we have
\[ \rho(x,y) \leq K \max(\rho(w,x), \rho(w,y)). \]
Without loss of generality, $\rho(w,x) \geq \rho(w,y)$ and hence
\[ \rho(x,y) \leq K \rho(w,x). \]
Further, we have
\[ \rho(z,w) \leq K \max(\rho(z,y), \rho(y,w)). \]
Combining these two inequalities, we conclude
\begin{equation*}
\begin{split}
\rho(x,y)\rho(z,w) & \leq K \rho(w,x)\rho(z,w)\\
& \leq K \frac{\epsilon}{\rho(y,z)}\rho(z,w)\\
& \leq \begin{cases} K^2 \epsilon & \text{if $\rho(z,y) \geq \rho(y,w)$}\\
K^2 \rho(w,x)\rho(y,w) \leq K^2 \epsilon & \text{if $\rho(z,y) < \rho(y,w)$}.
\end{cases}
\end{split}
\end{equation*}

We use this to show that $crt$ stays away from the corner points. Consider the triple
\[ (a,b,c) := ( \rho(w,x)\rho(y,z), \rho(w,y)\rho(x,z), \rho(w,z)\rho(x,y) ) \in \mathbb{R}^3.\]
The argument above shows that
\[ c \leq K^2\max(a,b). \]
Projecting $(a,b,c)$ to projective space, this implies that $(a:b:c) \notin \{ (a' : b' : 1) \in \mathbb{R}P^2 \vert a' < \frac{1}{K^2}, b' < \frac{1}{K^2} \}$, which is an open neighbourhood of $(0:0:1)$ in $\mathbb{R}P^2$. Since $(a:b:c) = crt(w, x, y, z)$, we found an open neighbourhood of $(0 : 0 : 1)$ that doesn't intersect with $\Ima(crt)$. Using analogous arguments, we find neighbourhoods of $(1 : 0 : 0)$ and $(0 : 1 : 0)$ that don't intersect with $\Ima(crt)$. This completes the proof of the Lemma.
\end{proof}

The other direction of the characterisation is based on the following result.

\begin{lem} \label{lem:cornerconditionatinfinity}
Let $\rho$ be a semi-metric on the set $Z$ such that $\rho$ has a point at infinity. Then $\rho$ is a quasi-metric if and only if its induced  M\"obius structure is a strong M\"obius structure.
\end{lem}

\begin{proof}
Lemma \ref{lem:cornercondition} immediately implies one direction of the proof. Suppose now $crt$ satisfies the (corner)-condition. We want to show that $\rho$ is a quasi-metric.

Denote the point at infinity with respect to $\rho$ by $\omega$. Let $x,y,z \in Z$. If two of the points are the same, or if one of the three points equals $\omega$, then the inequality for quasi-metrics is immediately satisfied. So assume $x,y,z$ are mutually different and different from $\omega$. Then $(x,y,z,\omega)$ is a non-degenerate quadruple and we can look at the cross-ratio-triple

\[ crt(x, y, z, \omega) = (\rho(x,y) : \rho(x,z) : \rho(y,z)). \]

Since $crt$ satisfies the (corner)-condition, we know that there is an open neighbourhood of $(1:0:0)$, independent of $x, y, z$, such that $crt(x, y, z, \omega)$ doesn't lie within that neighbourhood. We find $\epsilon > 0$, such that $crt(x, y, z, \omega) \notin N_{\epsilon}$, where
\[ N_{\epsilon} := \{Ê(1:b:c) \vert b,c \in (-\epsilon, \epsilon) \}, \; \epsilon > 0. \]
This implies that
\[ \max \left( \frac{\rho(x,z)}{\rho(x,y)}, \frac{\rho(y,z)}{\rho(x,y)} \right) \geq \epsilon \]
or equivalently
\[ \frac{1}{\epsilon}\max(\rho(x,z), \rho(z,y)) \geq \rho(x,y). \]
Thus, $\rho$ is a $\frac{1}{\epsilon}$-quasi-metric.

\end{proof}

Lemma \ref{lem:cornercondition} and Lemma \ref{lem:cornerconditionatinfinity} together with Buyalo's result that every M\"obius structure is induced by a semi-metric prove the first part of Proposition \ref{prop:MobiusStructuresQuasiMetrics}. We are left to prove the second part.

\begin{proof}[Proof of Proposition \ref{prop:MobiusStructuresQuasiMetrics}]

Let $\rho$ be a $K$-quasi-metric on $Z$ with a point at infinity. Denote the point at infinity by $\omega$. Choose a base point $o \in Z$. Now define for all $x, y, z \in Z$
\begin{equation*}
\begin{split}
& \lambda(z) := \max(1, \rho(z,o)),\\
& \tilde{\rho}(x, y) := \frac{ \rho(x, y) }{\lambda(x) \lambda(y)}.
\end{split}
\end{equation*}

By Proposition 5.3.6 from \cite{BS}, $\tilde{\rho}$ is a $K'^2$-quasi-metric for some $K' \geq K$. Furthermore,

\begin{equation*}
\begin{split}
\tilde{\rho}(x, y) & = \frac{\rho(x, y)}{ \lambda(x) \lambda(y) }\\
& \leq K \frac{\max(\rho(x, o),\rho(o, y))}{\lambda(x) \lambda(y)}\\
& \leq K
\end{split}
\end{equation*}

and thus, $\tilde{\rho}$ is a bounded quasi-metric on $Z$. A straightforward computation shows that $\rho$ and $\tilde{\rho}$ induce the same cross ratio and therefore, the same $M$.
\end{proof}


\section{The M\"obius topology} \label{sec:MobiusTopology}

Let $(Z, M)$ be a  M\"obius space. In order to construct a topology on $Z$, we will recall Buyalo's construction of a family of extended semi-metrics, each of which induces $M$. We will then use those semi-metrics to define a topology.

Since $M(w, x, y, z) \in \overline{L_4}$ is a triple, we write $M = (a, b, c)$ where $a, b, c: \mathcal{A} \rightarrow [-\infty, \infty]$ are the components of $M$. Condition 4) in the definition of  M\"obius structures now implies that for all non-degenerate triples $(\omega, \alpha, \beta)$ and $x, y \in Z \diagdown \{ \omega \}$, we have
\[ a(\alpha, x, \omega, \beta) + a(\alpha, \omega, y, \beta) - a( \alpha, x, y, \beta) = b(\alpha, x, y, \beta) - b(\alpha, x, \omega, \beta) - b(\alpha, \omega, y, \beta). \]
Therefore, denoting $A := ( \omega, \alpha, \beta)$, we can define
\[ \rho_A(x,y) := \begin{cases} 
0 & \text{if $x=y$},\\
e^{a(\alpha, x, \omega, \beta) + a(\alpha, \omega, y, \beta) - a(\alpha, x, y, \beta)} & \text{if $x \neq \beta$ and $y \neq \alpha$},\\
e^{b(\alpha, x, y, \beta) - b(\alpha, x, \omega, \beta) - b(\alpha, \omega, y, \beta)} & \text{if $x \neq \alpha$ and $y \neq \beta$}.
\end{cases}\]

In \cite{Buyalo}, Buyalo proved the following properties of $\rho_A$.

\begin{thm}[\cite{Buyalo}] \label{thm:BuyaloSemiMetrics}
Let $(Z,M)$ be a  M\"obius space, $\rho_A$ the map induced by $A$ for any non-degenerate triple $A$ in $Z$. Let $M_A$ be the cross difference induced by $\rho_A$. Then the following hold:

\begin{enumerate}

\item[1)] Every $\rho_A$ is an extended semi-metric on $Z$ i.\,e. $\rho_A$ is symmetric, non-negative and non-degenerate.

\item[2)] For all $x \neq \omega$, $\rho_{(\omega, \alpha, \beta)}(x, \omega) = \infty$. Moreover, $\rho_{(\omega, \alpha, \beta)}(\alpha, \beta) = 1$.

\item[3)] Let $A = (\omega, \alpha, \beta)$, $A' = (\omega, \beta, \alpha)$, $A'' = (\beta, \alpha, \omega)$. Then $\rho_A = \rho_{A'}$ and $\rho_{A''}(x,y) = \frac{\rho_A(x,y)}{\rho_A(x, \beta) \rho_A(\beta, y)}$.

\item[4)] Let $(\omega, \alpha, \beta, b)$ be a non-degenerate quadruple in $Z$. Then $\rho_{(\omega, \alpha, \beta)} = \lambda \rho_{(\omega, \alpha, b)}$ for some constant $\lambda > 0$.

\item[5)] For each non-degenerate triple $A$, $M_A = M$.

\end{enumerate}
\end{thm}

The following result is a straight-forward computation.

\begin{lem} \label{lem:InvolutionFormularho_A}
If $M$ is induced by a semi-metric $\rho$, then for every non-degenerate triple $A$ and for all $x \neq y$: $\rho_A(x,y) = \frac{\rho(x,y)}{\rho(x, \omega)\rho(\omega, y)}\frac{\rho(\alpha, \omega) \rho(\omega, \beta)}{\rho(\alpha, \beta)}.$
\end{lem}

\begin{proof}
Let $A$ be a non-degenerate triple and let $x, y \in Z$. Suppose, $x \neq \beta$ and $y \neq \alpha$.
\begin{equation*}
\begin{split}
\rho_A(x,y) & = e^{a(\alpha, x, \omega, \beta) + a(\alpha, \omega, y, \beta) - a(\alpha, x, y, \beta)}\\
& = cr(\alpha, \omega, \beta, x) \cdot cr(\alpha, y, \beta, \omega) \cdot cr(\alpha, y, \beta, x)^{-1}\\
& = \frac{ \rho(\alpha, \omega) \rho( \beta, x) \rho( \alpha, y) \rho(\beta, \omega) \rho(\alpha, \beta) \rho(x,y) }{ \rho(\alpha, \beta) \rho(\omega, x) \rho(\alpha, \beta) \rho(y, \omega) \rho(\alpha, y) \rho(\beta, x) }\\
& = \frac{\rho(x,y)}{\rho(x,\omega)\rho(\omega,y)} \frac{\rho(\alpha, \omega) \rho(\omega, \beta)}{\rho(\alpha, \beta)}.
\end{split}
\end{equation*}
The case when $x \neq \alpha$ and $y \neq \beta$ is analogous.
\end{proof}

We see that $\{Ê\rho_A \}_A$ is a family of extended semi-metrics that can be constructed from a  M\"obius structure $M$. In \cite{Buyalo}, these semi-metrics are used to define the following topology.

Let $A = (\omega, \alpha, \beta)$ be a non-degenerate triple. For $y \in Z \diagdown \{ \omega \}$ and $r > 0$, define

\[ B_{A,r}(y) := \{ x \in Z \vert \rho_A(x,y) < r \}Ê\]

to be the {\it open ball} around $y$ of radius $r$ with respect to $\rho_A$. We take the family of all open balls for all non-degenerate triples $A$, all positive radii $r$ and all points $y \in Z \diagdown \{ \omega \}$ as a subbasis to define a topology $\mathcal{T}_M$ on $Z$. This is the {\it topology on $Z$ induced by $M$}. From now on, whenever we speak of a  M\"obius space $(Z,M)$ we assume it to be endowed with the topology induced by $M$, unless stated otherwise.\\

\begin{lem}
Consider $[0, \infty]$ with the topology where open neighbourhoods of $\infty$ are complements of compact sets in $[0, \infty)$ and open neighbourhoods of other points are just the standard euclidean open neighbourhoods. Let $(Z,M)$ be a  M\"obius space, $A$ a non-degenerate triple in $Z$ and $y \in Z$. Then the maps $\rho_A( \cdot, y), \rho_A(y, \cdot) : Z \rightarrow [0, \infty]$ are continuous with respect to $\mathcal{T}_M$.
\end{lem}

\begin{proof}
First note that if $y = \omega$, $\rho_A( \cdot, y) \equiv \infty$ is constant and hence continuous. If $y \neq \omega$, we start by defining the set

\[ C_{A,r}(y) := \{ x \in Z \vert \rho_A(x,y) > r \} \]

which can be thought of as the complement of a `closed' ball (again $y \neq \omega$). Let $(a,b)$ be an open interval in $\mathbb{R}$ (possibly unbounded) and consider the map $f:= \rho_A(\cdot, y)$ for some fixed $y \neq \omega$. Then $f^{-1}((a,b)) = B_{A,b}(y) \cap C_{A,a}(y)$ and continuity of $f$ follows, if $C_{A,a}(y)$ is open for all $a \geq 0$.

By Theorem \ref{thm:BuyaloSemiMetrics}, we know that for any non-degenerate triple $(\omega, \alpha, y)$ and every $x \in Z \diagdown \{ y, \omega \}$
\begin{equation*}
\begin{split}
& \rho_{(\omega, \alpha, y)} (x, y) \rho_{(y, \alpha, \omega)}(x, \omega) = 1,\\
& \rho_{(\omega, \alpha, \beta)}(x,y) = \lambda \rho_{(\omega, \alpha, o)}(x,y).
\end{split}
\end{equation*}
Therefore, we see that
\[ \rho_{(y, \alpha, o)}(x, \omega) = \lambda \rho_{(y, \alpha, \omega)}(x,\omega) = \frac{\lambda}{\rho_{(\omega, \alpha, y)}(x, y)} = \frac{\lambda}{\mu \rho_{(\omega, \alpha, \beta)}(x,y)} \]

for $y,$ $\omega,$ $\alpha,$ $\beta,$ $o$ mutually different and $\lambda,$ $\mu > 0$ depending only on $\alpha,$ $\omega,$ $y,$ $o$ and $\alpha,$ $\beta,$ $\omega,$ $y$ respectively. This immediately implies that $B_{(\omega, \alpha, \beta),r}(y) = C_{(y, \alpha, o), \frac{\lambda}{\mu r}}(\omega)$ for some $\lambda,$ $\mu > 0$ (notice that the points $\omega$ and $y$ behave nicely). Since this is true for all $\omega,$ $\alpha,$ $\beta,$ $y,$ $o$ and $r$ as above, we see that $C_{A,r}(y)$ is open for all non-degenerate triples $A$, all $r > 0$ and all $y \in Z$. This implies the Lemma.
\end{proof}

\begin{rem}
Notice that the proof of the continuity of $\rho_A$ relies on the fact that we take the open balls of {\it all} semi-metrics $\rho_A$. It is not sufficient to take just one - or some - of the non-degenerate triples. Only collectively can they define a topology such that $\rho_A(\cdot, y)$ is continuous. In particular, the involution plays a critical role. The following example illustrates how a the topology induced by a single quasi-metric does not have this continuity-property.
\end{rem}

\begin{exam} \label{exam:Continuity}
Let $X = [0, 1]$ and define
\[ \rho(x,y) := \begin{cases} 
\vert x - y \vert & \text{if $\vert x - y \vert < 1$},\\
2 \vert x - y \vert & \text{if $\vert x - y \vert \geq 1$}.
\end{cases}\]
Since for all $x, y, z \in X$, we have
\[ \rho(x, y) \leq 2 \vert x - y \vert \leq 2 ( \vert x - z \vert + \vert z - y \vert ) \leq 4 \max( \rho(x,z), \rho(z,y) ), \]
we see that $\rho$ is a $4$-quasi-metric. Consider the topology generated by the `open balls' $B_r(x) := \{ y \in Z \vert \rho(x,y) < r \}$ and the sequence $x_n := 1 - \frac{1}{n}$. The sequence $x_n$ converges to $1$ in the topology induced by $\rho$. However,
\[ \rho(0, x_n) = 1 - \frac{1}{n} \xrightarrow{n \rightarrow \infty} 1 \neq 2 = \rho(0,1) \]
and therefore, $\rho$ is not continuous with respect to the topology it induces. This is in significant contrast to metric spaces or the maps $\rho_A$ with the M\"obius topology.
\end{exam}

\begin{lem}
The topological space $(Z,\mathcal{T}_M)$ is Hausdorff.
\end{lem}

\begin{proof}
Let $x,y \in Z$ be two different points. Choose a point $\alpha \in Z \diagdown \{x,y\}$. We know that for every $z \in Z$,

\[ \rho_{(y,\alpha, x)}(x,z) = \frac{1}{\rho_{(x,\alpha, y)}(y,z)}. \]

Therefore, the intersection of the two open balls $B_{(y, \alpha, x),1}(x), B_{(x,\alpha,y),1}(y)$ is empty. This completes the proof.

\end{proof}

Consider two  M\"obius spaces $(Z,M), (Z',M')$. We want to have a notion of maps that are compatible with the M\"obius structures.

\begin{mydef} \label{def:MobiusEquivalence}
Let $(Z,M), (Z',M')$ be  M\"obius spaces. A map $f : Z \rightarrow Z'$ is called a {\it M\"obius map} if and only if for every admissible quadruple $(w,x,y,z) \in \mathcal{A}$, we have
\[ M(w, x, y, z) = M(f(w), f(x), f(y), f(z)). \]
If a M\"obius map $f$ is bijective, we call it a {\it M\"obius equivalence}.
\end{mydef}

\begin{lem} \label{lem:homeomorphism}
Let $(Z,M), (Z',M')$ be two  M\"obius spaces and $f : Z \rightarrow Z'$ a M\"obius equivalence. Then $f$ is a homeomorphism when we equip $Z$ and $Z'$ with their respective M\"obius topology.
\end{lem}

\begin{proof}
Let $A$ be a non-degenerate triple in $Z$. Since $f$ is a bijection, it sends $A$ to a non-degenerate triple denoted $f(A)$ in $Z'$. Looking at the definition of the semi-metric $\rho_A$, we immediately see that, since $f$ preserves the  M\"obius structure, we have for all $x,y \in Z$

\[ \rho_A(x,y) = \rho_{f(A)}(f(x),f(y)). \]

Thus, the map $f$ sends an open ball $B_{A,r}(x)$ in $Z$ to the open ball $B_{f(A),r}(f(x))$ in $Z'$ and a subbasis of $\mathcal{T}_M$ to a subbasis of $\mathcal{T}_{M'}$. The same is true for $f^{-1}$ which proves the Lemma.

\end{proof}

Classically, M\"obius structures arise in the study of metric spaces. When a M\"obius structure arises from a metric, the topology constructed above coincides with the topology induced by the metric.

\begin{thm} \label{thm:metrictopology}
Let $(Z,\rho)$ be a metric space. Denote the topology on $Z$ induced by $\rho$ by $\mathcal{T}_\rho$ and the induced  M\"obius structure by $M$. Let $\mathcal{T}_M$ be the topology induced by $M$ and $\{ \rho_A \}_A$ the family of semi-metrics induced by $M$. Then $\mathcal{T}_{\rho} = \mathcal{T}_M$.
\end{thm}

\begin{proof}
Since $Z$ is a metric space, Lemma \ref{lem:InvolutionFormularho_A} tells us that for all non-degenerate triple $A$ and for all $x \neq y$, $\rho_A(x,y) = \frac{\rho(x,y)}{\rho(x, \omega)\rho(\omega,y)}\frac{\rho(\alpha, \omega)\rho(\omega, \beta)}{\rho(\alpha, \beta)}$. In particular, $\rho_A(x, y)$ is continuous in $x$ with respect to $\mathcal{T}_{\rho}$ as long as $x \in Z \diagdown \{ \omega \}$.

We need to show that the open balls in $\rho$ are open with respect to $\mathcal{T}_M$ and that the open balls with respect to the $\rho_A$ are open with respect to $\mathcal{T}_\rho$. We denote by

\[ B_{s}(y) := \{ x \in Z \vert \rho(x,y) < s \} \]

the open ball of radius $s$ with respect to $\rho$ and by

\[ B_{A,s}(y) := \{ x \in Z \vert \rho_A(x,y) < s \} \]

the open ball of radius $s$ with respect to $\rho_A$. These sets generate $\mathcal{T}_\rho$ respectively $\mathcal{T}_M$.\\

We first show that $B_{A,r}(y)$ is open with respect to $\mathcal{T}_\rho$ for all non-degenerate triples $A, r > 0$ and $y \in Z \diagdown \{ \omega \}$. Let $x \in B_{A,r}(y)$, i.e.\,$\rho_A(x,y) < r$. Since $\rho_A$ is continuous with respect to $\mathcal{T}_{\rho}$, there exists some $\epsilon > 0$, such that $B_{\epsilon}(x) \subset B_{A,r}(y)$. We conclude that $B_{A,r}(y)$ is open in $\mathcal{T}_{\rho}$ and that $\mathcal{T}_{\rho}$ is finer than $\mathcal{T}_M$.\\

In order to show that $\mathcal{T}_M$ is finer than $\mathcal{T}_\rho$, we consider the open ball $B_{r}(y)$ for $r > 0, y \in Z$. Let $x \in B_{r}(y)$. Since $\rho$ is a metric, there exists a smaller ball around $x$ contained in $B_{r}(y)$, i.e. there exists $r' < r$ such that $B_{r'}(x) \subset B_{r}(y)$. Replacing $r'$ by $\epsilon$, it is now enough to show that for every $\epsilon > 0$, we can find $\delta > 0$ and a non-degenerate triple $A$ such that $B_{A,\delta}(x) \subset B_{\epsilon}(x)$.

Choose $\omega, \alpha \in Z$ such that $A := (\omega, \alpha, x)$ is non-degenerate. We claim that there exist $\delta > 0$ and $C > 0$, such that for all $z \in B_{A,\delta}(x)$, $\rho(z, \omega) < C$. Suppose not. Then we find a sequence $z_n$ such that $\rho_A(z_n, x) \rightarrow 0$ and $\rho(z_n, \omega) \rightarrow \infty$. However,
\begin{equation*}
\begin{split}
0 \leftarrow \rho_A(z_n, x) & = \frac{ \rho(z_n, x) \rho(\alpha, \omega) }{\rho(z_n, \omega) \rho(\alpha, x)}\\
& \geq \frac{  \rho(z_n, \omega) - \rho(\omega, x)  }{\rho(z_n, \omega)} \frac{\rho(\alpha, \omega) }{\rho(\alpha,x)} \rightarrow \frac{\rho(\alpha, \omega)}{\rho(\alpha, x)} \neq 0.
\end{split}
\end{equation*}

Let $0 < \delta' < \min \left( \epsilon \frac{Ê\rho(\alpha, \omega) }{ C \rho(\alpha, x) }, \delta \right)$ and $z \in B_{A, \delta'}(x)$. Then
\begin{equation*}
\begin{split}
\rho(z,x) & = \frac{ \rho(z,x) \rho(\alpha, \omega) }{ \rho(z,\omega) \rho(\alpha,x) } \frac{ \rho(z,\omega) \rho(\alpha, x) }{ \rho(\alpha, \omega) }\\
& \leq \rho_A(z,x) \frac{ C \rho(\alpha, x) }{ \rho(\alpha, \omega) }\\
& \leq \epsilon 
\end{split}
\end{equation*}
In other words, $B_{A,\delta'}(x) \subseteq B_{\epsilon}(x)$. This implies that balls with respect to $\rho$ are open with respect to $\mathcal{T}_M$. Hence $\mathcal{T}_M$ is finer than $\mathcal{T}_\rho$, which completes the proof of Theorem \ref{thm:metrictopology}.

\end{proof}

\begin{rem}
This proof easily extends to extended metric spaces which have a point at infinity: Let $\infty$ denote the point at infinity in the metric space $(Z,\rho)$. Then, for any non-degenerate triple $A = (\infty, \alpha, \beta)$, we have $\rho_A = \lambda \rho$ for some positive number $\lambda$. This immediately implies that $\mathcal{T}_{\rho} \subseteq \mathcal{T}_M$. To prove equality, one modifies the proof provided above.
\end{rem}

Applying Lemma \ref{lem:homeomorphism} in the context of Theorem \ref{thm:metrictopology} immediately yields the following

\begin{cor}
Let $(Z,\rho), (Z',\rho')$ be -- possibly extended -- metric spaces, $M_{\rho}, M_{\rho'}$ the induced M\"obius structures and $f : (Z, M_\rho) \rightarrow Z'$ a M\"obius equivalence. Then $f$ is a homeomorphism with respect to the metric topologies $\mathcal{T}_{\rho}, \mathcal{T}_{\rho'}$.
\end{cor}

\begin{proof}
We know from Lemma \ref{lem:homeomorphism} that $f$ is a homeomorphism with respect to the topologies $\mathcal{T}_M, \mathcal{T}_{M'}$. By Theorem \ref{thm:metrictopology}, the M\"obius topologies and the metric topologies coincide, i.e. $\mathcal{T}_M = \mathcal{T}_\rho$ and $\mathcal{T}_{M'} = \mathcal{T}_{\rho'}$. The statement follows.
\end{proof}




\section{Cauchy-sequences and completeness} \label{sec:CauchySequences}

The next two sections are devoted to the notion of Cauchy-sequences. We show how to define Cauchy-sequences on strong M\"obius spaces in a way that is compatible with the situation when the strong M\"obius structure is induced by a metric space. In the next section, we show how to construct a completion under an additional symmetry assumption.

Let $(Z,\rho)$ be a metric space, $M$ its induced  strong M\"obius structure. We recall that a Cauchy sequence -- in its usual sense on a metric space -- is a sequence $(x_n)_n$ in $Z$ such that for all $\epsilon > 0$ there exists a natural number $N_{\epsilon}$ such that for all $m,n \geq N_{\epsilon}$, we have $\rho(x_m,x_n) < \epsilon$. Our goal is to generalise this notion to strong M\"obius spaces. It may be tempting to simply generalise the statement above to quasi- and semi-metrics and use that as a definition, but since a  M\"obius structure can be induced by many different semi-metrics, a definition relying only on the M\"obius structure itself is more desirable.\\

Before we formulate the key insight, we need some notation. Let $\rho$ be a, possibly extended, semi-metric that induces $M$. If $\rho$ has a point at infinity, we denote that point by $\omega$. We denote $(y \vert z) := -\ln(\rho(y,z))$ for all $y, z \in Z$. Further, consider a sequence $(x_{n,m})_{n,m \in \mathbb{N}}$ in $Z$. We say that $\lim \limits_{n,m \rightarrow \infty} x_{n,m} = y$, if and only if for all $\epsilon > 0$ there exists an $N_{\epsilon}$ such that for all $n,m \geq N_{\epsilon}$, we have $\rho(x_{n,m}, y) < \epsilon$.

In what follows, we will often consider a sequence $(x_n)_n$ and a pair of points $y, z \in Z \diagdown \{ \omega \}$ such that $y \neq z$ and $\rho(x_n, y)$ and $\rho(x_n, z)$ both don't converge to zero. Given a sequence $(x_n)_n$, we will refer such a pair $y,z$ as a {\it good pair}.

Recall that we write $M = (a,b,c)$, where $a$, $b$, $c$ denote the components of $M$. We can now characterize Cauchy sequences in terms of the  M\"obius structure.

\begin{lem} \label{lem:CauchySequences}
Let $(Z,\rho)$ be a metric space, $(x_n)_{n \in \mathbb{N}}$ a sequence in $Z$. The following are equivalent:
\begin{enumerate}

\item[1)] The sequence $(x_n)_n$ is either a Cauchy sequence, or $\rho(x_n, y) \xrightarrow{n \rightarrow \infty} \infty$ for all $y \in Z$.

\item[2)] There exists a good pair $y, z \in Z$, such that $\lim \limits_{n,m \rightarrow \infty} crt(x_n, x_m, y, z) = (0:1:1)$.

\item[3)] There exists a good pair $y, z \in Z$, $\lim \limits_{n,m \rightarrow \infty} c(x_n, x_m, y, z) = -\infty$.

\end{enumerate}

Further, if 1 holds, then 2 and 3 hold for all good pairs $y, z \in Z$. In addition, 2 holds for a good pair $y,z$ if and only if 3 holds for the same good pair $y,z$.
\end{lem}

The equivalence of 1) and 2) is stated in Lemma 2.2 of \cite{BeyrerSchroeder}. Furthermore, it is easy to see from the proof that 1) implies 2) for every good pair. We are left to prove ``2) $\Rightarrow$ 3)'' and ``3) $\Rightarrow$ 1)''. For this, we require an auxiliary result. Since it is our goal to generalise Cauchy sequences beyond the realm of metric spaces, we will formulate this result in a more general context.

\begin{lem} \label{lem:AuxiliaryforCauchy}
Let $(Z,M)$ be a  strong M\"obius structure and $\rho$ a quasi-metric that induces $M$. Let $(x_n)_n$ be a sequence in $Z$ and suppose there exists a good pair $y,z \in Z$ such that $c(x_n, x_m, y, z) \xrightarrow{n,m \rightarrow \infty} -\infty$. Then one of the following two statements holds:

\begin{enumerate}

\item[a)] For every $x \in Z \diagdown \{ \omega \}$, there exists some $B_x > 0$, such that $\rho(x_n,x) < B_x$ for all $n$. Furthermore, $\rho(x_n,x_m) \xrightarrow{n,m \rightarrow \infty} 0$. We say that $x_n$ is bounded.

\item[b)] For every $x \in Z \diagdown \{ \omega \}$, we have $\rho(x_n, x) \xrightarrow{n \rightarrow \infty} \infty$. We say that $x_n$ diverges to infinity and write $x_n \rightarrow \infty$.

\end{enumerate}
\end{lem}

Lemma \ref{lem:AuxiliaryforCauchy} is a generalization of the statement 3) $\Rightarrow$ 1) in Lemma \ref{lem:CauchySequences}.

\begin{rem}
Lemma \ref{lem:CauchySequences} and Lemma \ref{lem:AuxiliaryforCauchy} also hold for extended metric spaces. One can prove 1) $\Rightarrow$ 2) for the case $y = \omega$ separately (and by symmetry, the same proof works for $z = \omega$). The proof of 2) $\Rightarrow$ 3) that we see below immediately generalises to extended metric spaces. For 3) $\Rightarrow$ 1), we can use the fact that by Lemma \ref{lem:AuxiliaryforCauchy}, this statement also holds for quasi-metrics. If $y = \omega$ for a given quasi-metric, we can perform involution of $\rho$ at any point $x \in Z \diagdown \{ y,z \}$. This provides us with a quasi-metric that induces the same  strong M\"obius structure, but neither $y$ nor $z$ lies at infinity.
\end{rem}

\begin{proof}[Proof of Lemma \ref{lem:AuxiliaryforCauchy}]

Let $(x_n)_n$ be a sequence in the  strong M\"obius space $(Z,M)$, let $\rho$ be a quasi-metric that induces $M$ and let $y,z$ be a good pair such that $c(x_n,x_m,y,z) \xrightarrow{n,m \rightarrow \infty} -\infty$. By definition of the  M\"obius structure induced by $\rho$, we can write

\[c(x_n, x_m, y, z) = (x_n | y) + (x_m | z) - (x_n | x_m) - (y | z) = \ln \left( \frac{\rho(x_n,x_m)\rho(y,z)}{\rho(x_n,y)\rho(x_m,z)} \right). \]

Using this equality, the statement $c(x_n, x_m,y,z) \xrightarrow{n,m \rightarrow \infty} -\infty$ becomes equivalent to

\begin{equation} \label{eq:coefficientc}
\frac{\rho(x_n,x_m)\rho(y,z)}{\rho(x_n,y)\rho(x_m,z)} \xrightarrow{n,m \rightarrow \infty} 0.
\end{equation}\\

We will distinguish between two cases, which will turn out to be exactly the distinction between Case a) and Case b). Suppose there exists some $x \in Z \diagdown \{ \omega \}$ and some constant $B > 0$ such that $\rho(x_n, x) < B$ for all $n$. We want to show that we are in Case a).

Since $\rho$ is a quasi-metric, we have that for all $x' \in Z \diagdown \{ \omega \}$,
\[ \rho(x_n, x') \leq K \max(\rho(x_n, x), \rho(x,x')) \leq K \max(B, \rho(x,x')).\]
Therefore, we see that $\rho(x_n,x')$ is bounded for all $x' \in Z \diagdown \{ \omega \}$. In particular, $\rho(x_n,y), \rho(x_n,z)$ are both bounded by some constant $B > 0$. We obtain
\begin{equation*}
\begin{split}
\frac{\rho(x_n,x_m)\rho(y,z)}{\rho(x_n,y)\rho(x_m,z)} & \geq \rho(x_n,x_m) \frac{\rho(y,z)}{B^2}.
\end{split}
\end{equation*}

Since the left-hand-side of this equation goes to zero by assumption, the right-hand-side has to go to zero as well. Hence we see that $\rho(x_n,x_m) \xrightarrow{n,m \rightarrow \infty} 0$.\\

We are left to show that we end up in Case b), whenever there is no $x \in Z \diagdown \{ \omega \}$ such that $\rho(x_n,x)$ is bounded. Suppose $\rho(x_n,x)$ is unbounded for all $x \in Z \diagdown \{ \omega \}$. Then there exists a subsequence $(x_{n_i})_i$ of $(x_n)_n$ such that $\rho(x_{n_i},x) \rightarrow \infty$ for one (and hence all, since $\rho$ is a quasi-metric) $x \in Z \diagdown \{ \omega \}$.

Suppose by contradiction that $\rho(x_n,x)$ does not converge to infinity for one and hence all $x \in Z \diagdown \{ \omega \}$. Then we find another subsequence $(x_{m_j})_j$ of $(x_n)_n$, which is bounded. In particular, we find a constant $B > 0$ such that
\begin{equation*}
\begin{split}
& \rho(x_{m_j}, y) \leq B,\\
& \rho(x_{m_j}, z) \leq B
\end{split}
\end{equation*}
for all $j$. From our treatment of Case a), we know that for this subsequence, $\rho(x_{m_j}, x_{m_{j'}}) \xrightarrow{j,j' \rightarrow \infty} 0$. In particular, we find a number $J$ such that for all $j,j' \geq J$, we have
\[ \rho(x_{m_j},x_{m_{j'}}) < 1. \]

Now, we estimate the distance between the two subsequences $(x_{m_j})_j$, $(x_{n_i})_i$. For this, we need to to take $x_{m_J}$ as an auxiliary point. Since $x_{n_i}$ diverges to infinity, there is a number $I$ such that $\rho(x_{m_J}, x_{n_i}) > \max(K, K \cdot B)$ for all $i \geq I$. Now we use the fact that $\rho$ is a quasi-metric to get that for all $i \geq I, j \geq J$ we have
\[ \max(K, K \cdot B) \leq \rho(x_{m_J}, x_{n_i}) \leq K \max( \rho(x_{m_J}, x_{m_{j}}), \rho(x_{m_{j}}, x_{n_i}) ) = K \rho(x_{m_{j}}, x_{n_i}), \]
where the last equality follows from the fact that $\rho(x_{m_J}, x_{m_j}) < 1$ for all $j \geq J$. Now consider for $i \geq I, j \geq J$
\begin{equation*}
\begin{split}
\frac{\rho(x_{m_j},x_{n_i})\rho(y,z)}{\rho(x_{m_j},y)\rho(x_{n_i},z)} & \geq \frac{\rho(x_{m_j},x_{n_i})\rho(y,z)}{B \rho(x_{n_i},z)}\\
& \geq \frac{\rho(x_{m_j},x_{n_i})\rho(y,z)}{B K \max( \rho(x_{n_i}, x_{m_j}), \rho(x_{m_j},z) )}\\
& = \frac{\rho(x_{m_j},x_{n_i})\rho(y,z)}{B K \rho(x_{n_i}, x_{m_j}) }\\
& = \frac{\rho(y,z)}{BK},
\end{split}
\end{equation*}
where in the second-to-last step we use the fact that $\rho(x_{n_i}, x_{m_j}) \geq \max(1,B) \geq B \geq \rho(x_{m_j},z)$ for all $i \geq I, j \geq J$. This inequality shows that $\frac{\rho(x_{m_j},x_{n_i})\rho(y,z)}{\rho(x_{m_j},y)\rho(x_{n_i},z)}$ is bounded from below by a positive constant. But by assumption, $\frac{\rho(x_{m_j},x_{n_i})\rho(y,z)}{\rho(x_{m_j},y)\rho(x_{n_i},z)}$ converges to zero, a contradiction. We see that, if a subsequence $(x_{n_i})_i$ diverges to infinity, the sequence $(x_n)_n$ has to diverge to infinity as well. Thus, we are in Case b), which completes the proof.

\end{proof}

\begin{proof}[Proof of Lemma \ref{lem:CauchySequences}]
Let $(Z,\rho)$ be a non-extended metric space, $(x_n)_n$ a sequence in $Z$ and $y,z \in Z$ such that $\lim \limits_{n \rightarrow \infty} x_n \neq y,z$.\\

1) $\Rightarrow$ 2): Instead of proving just 1) $\Rightarrow$ 2), which follows directly from \cite{BeyrerSchroeder}, we will also prove the second part of the Lemma, i.e. we will prove that $\lim\limits_{n,m \rightarrow \infty} crt(x_n, x_m, y, z) = (0:1:1)$ for all good pairs $y,z$.\\

Step 1: We start by proving that for every Cauchy sequence, we have
\[ \lim_{n,m \rightarrow \infty} crt(x_n,x_m, y,z) = (0:1:1). \]

Suppose $(x_n)$ is a Cauchy sequence. Note that this implies that $\rho(x_n, x)$ converges for all $x \in Z$. Let $\epsilon > 0$. We find some $N_{\epsilon} \in \mathbb{N}$ such that for all $n, m \geq N_{\epsilon}$, we have $\rho(x_n, x_m) < \epsilon$. Since $y,z$ is a good pair, we can choose $\epsilon$ sufficiently small such that there is an $N_{\epsilon}$ such that additionally, $\rho(x_n,y),\rho(x_n,z) > \epsilon^{\frac{1}{4}}$ for all $n \geq N_{\epsilon}$. Therefore, we get
\begin{equation*}
\begin{split}
\frac{\rho(x_n,x_m)\rho(y,z)}{\rho(x_n,y)\rho(x_m,z)} & < \frac{\epsilon \rho(y,z)}{\rho(x_n,y)\rho(x_m,z)}\\
& < \frac{\epsilon}{\sqrt{\epsilon}}\rho(y,z)\\
& = \sqrt{\epsilon}\rho(y,z).
\end{split}
\end{equation*}

Thus we see that $\frac{\rho(x_n,x_m)\rho(y,z)}{\rho(x_n,y)\rho(x_m,z)} \xrightarrow{n,m \rightarrow \infty} 0$. For symmetry reasons, we immediately see that also $\frac{\rho(x_n,x_m)\rho(y,z)}{\rho(x_n,z)\rho(x_m,y)} \xrightarrow{n,m \rightarrow \infty} 0$.

We are left to show that $\frac{\rho(x_n,y)\rho(x_m,z)}{\rho(x_n,z)\rho(x_m,y)} \xrightarrow{n,m \rightarrow \infty} 1$ in order to prove that $crt(x_n,x_m,y,z) \xrightarrow{n,m \rightarrow \infty} (0:1:1)$. Since $y, z$ is a good pair, we have
\begin{equation*}
\begin{split}
\frac{\rho(x_n,y)\rho(x_m,z)}{\rho(x_n,z)\rho(x_m,y)} & \leq \frac{\rho(x_n,y)(\rho(x_n,z) + \rho(x_n,x_m))}{\rho(x_n,z)(\rho(x_n,y) - \rho(x_n,x_m))}\\
& = \frac{ 1 + \frac{ \rho(x_n, x_m) }{ \rho(x_n, z) } }{ 1 - \frac{ \rho(x_n, x_m) }{ \rho(x_n,y) } } \xrightarrow{ n, m \rightarrow \infty} 1
\end{split}
\end{equation*}
and
\begin{equation*}
\begin{split}
\frac{\rho(x_n,y)\rho(x_m,z)}{\rho(x_n,z)\rho(x_m,y)} & \geq \frac{\rho(x_n,y)(\rho(x_n,z) - \rho(x_n,x_m))}{\rho(x_n,z)(\rho(x_n,y) + \rho(x_n,x_m))}\\
& = \frac{ 1 - \frac{ \rho(x_n, x_m) }{ \rho(x_n, z) } }{ 1 + \frac{ \rho(x_n, x_m) }{ \rho(x_n,y) } } \xrightarrow{ n, m \rightarrow \infty} 1.
\end{split}
\end{equation*}

It follows that $\frac{\rho(x_n,y)\rho(x_m,z)}{\rho(x_n,z)\rho(x_m,y)} \xrightarrow{n,m \rightarrow \infty} 1$ and hence $crt(x_n,x_m,y,z) \xrightarrow{n,m \rightarrow \infty} (0:1:1)$. Note that we relied on the triangle-inequality for this part of the proof.\\

Step 2: We show that if $(x_n)$ diverges to infinity, we get
\[ \lim_{n,m \rightarrow \infty} crt(x_n,x_m, y,z) = (0:1:1). \]

Suppose that $\rho(x_n, x) \rightarrow \infty$ for all $x \in Z$ as $n$ goes to infinity (except for the point $x \in Z$ that may lie at infinity). Then, for any $y, z \in Z$ that do not lie at infinity, we have
\begin{equation*}
\begin{split}
\frac{\rho(x_n,x_m)\rho(y,z)}{\rho(x_n,y)\rho(x_m,z)} & \leq \frac{(\rho(x_n, y) + \rho(y, x_m))\rho(y,z)}{\rho(x_n,y)\rho(x_m,z)}\\
& = \frac{\rho(y,z)}{\rho(x_m,z)} + \frac{\rho(x_m,y)\rho(y,z)}{\rho(x_n,y)\rho(x_m,z)}\\
& \leq \frac{\rho(y,z)}{\rho(x_m,z)} + \frac{(\rho(x_m,z) + \rho(z,y))\rho(y,z)}{\rho(x_n,y)\rho(x_m,z)}\\
& = \frac{\rho(y,z)}{\rho(x_m,z)} + \frac{\rho(y,z)}{\rho(x_n,y)} + \frac{\rho(y,z)^2}{\rho(x_n,y)\rho(x_m,z)}\\
& \xrightarrow{n, m \rightarrow \infty} 0.
\end{split}
\end{equation*}

We are left to show that $\frac{\rho(x_n,y)\rho(x_m,z)}{\rho(x_n,z)\rho(x_m,y)} \xrightarrow{n,m \rightarrow \infty} 1$. For this, we do the following estimate.
\begin{equation*}
\begin{split}
\frac{\rho(x_n,y)\rho(x_m,z)}{\rho(x_n,z)\rho(x_m,y)} & \leq \frac{(\rho(x_n,z) + \rho(y,z))(\rho(x_m,y) + \rho(y,z))}{\rho(x_n,z)\rho(x_m,y)}\\
& = 1 + \frac{\rho(y,z)}{\rho(x_n,z)} + \frac{\rho(y,z)}{\rho(x_m,y)} + \frac{\rho(y,z)^2}{\rho(x_n,z)\rho(x_m,y)}\\
& \xrightarrow{n, m \rightarrow \infty} 1.
\end{split}
\end{equation*}
In the same way, we have
\begin{equation*}
\begin{split}
\frac{\rho(x_n,y)\rho(x_m,z)}{\rho(x_n,z)\rho(x_m,y)} & \geq \frac{(\rho(x_n,z) - \rho(y,z))(\rho(x_m,y) - \rho(y,z))}{\rho(x_n,z)\rho(x_m,y)}\\
& = 1 - \frac{\rho(y,z)}{\rho(x_n,z)} - \frac{\rho(y,z)}{\rho(x_m,y)} + \frac{\rho(y,z)^2}{\rho(x_n,z)\rho(x_m,y)}\\
& \xrightarrow{n, m \rightarrow \infty}Ê1.
\end{split}
\end{equation*}
From these two estimates, we conclude that $\frac{\rho(x_n,y)\rho(x_m,z)}{\rho(x_m,y)\rho(x_n,z)} \xrightarrow{n,m \rightarrow \infty} 1$. This concludes the proof of step 2 and the proof that 1) $\Rightarrow$ 2).\\

2) $\Rightarrow$ 3): Recall that, by definition, $c(w, x, y, z) = \ln \left( \frac{ \rho(w, x) \rho(y,z) }{ \rho(w,y) \rho(x,z) } \right)$, which is a continuous map with respect to the metric topology. In particular, if $crt(w, x, y, z) \rightarrow (0 : 1 : 1)$, then $\ln \left( \frac{ \rho(w, x) \rho(y,z) }{ \rho(w,y) \rho(x,z) } \right) \rightarrow -\infty$. We see that 2) $\Rightarrow$ 3). In particular, if 2) holds for a given pair $y,z$ then 3) holds for the same pair $y,z$.\\

3) $\Rightarrow$ 1): This is a special case of Lemma \ref{lem:AuxiliaryforCauchy}. Since we have seen that 1) $\Rightarrow$ 2) for all good pairs $y,z$ we also see that, if 3) holds for a good pair $y,z$ then 2) holds for the same good pair $y,z$. This concludes the proof of Lemma \ref{lem:CauchySequences}
\end{proof}

Among other things, Lemma \ref{lem:CauchySequences} tells us that for metric spaces, we only need to find one good pair $y, z$ that satisfies condition 2 or 3 to get the same condition for all good pairs $y, z$ that aren't the limit of $(x_n)_n$. It would be good to have the same condition in any strong M\"obius space that isn't necessarily induced by a metric. Then we could define a sequence in a strong M\"obius space to be a Cauchy sequence if for one good pair $y,z$ and hence all nice pairs, we have $crt(x_n,x_m,y,z) \rightarrow (0:1:1)$, which would be much easier to check in practice than if we had to check all good pairs. The next lemma tells us, that this is actually true in the case of condition 3.

\begin{lem} \label{lem:Conditionc}
Let $(Z,M)$ be a  strong M\"obius space. Let $(x_n)_n$ be a sequence in $Z$. Suppose there is a good pair $y, z$ such that

\[c(x_n,x_m,y,z) \xrightarrow{n,m \rightarrow \infty} -\infty.\]

Then the same holds for all good pairs $y', z' \in Z$.
\end{lem}

\begin{proof}

Let $\rho$ be a quasi-metric that induces $M$. By Lemma \ref{lem:AuxiliaryforCauchy}, we know that $(x_n)$ is either bounded or diverges to infinity. Let $y',z'$ be a good pair. As we have seen in the proofs of Lemma \ref{lem:CauchySequences} and \ref{lem:AuxiliaryforCauchy}, we get the right convergence of $c(x_n,x_m,y',z')$ if $\frac{\rho(x_n,x_m)\rho(y',z')}{\rho(x_n,y')\rho(x_m,z')}$ converges to zero.\\

Case 1: Suppose $(x_n)_n$ is bounded. Since $y',z'$ is a good pair, we find some $\epsilon > 0$ and a subsequence $(x_{n_i})_i$ such that $\rho(x_{n_i},y') \geq \epsilon$ for all $i$. From Lemma \ref{lem:AuxiliaryforCauchy}, we know that $\rho(x_n, x_m) \xrightarrow{n,m \rightarrow \infty} 0$ and we find a number $N$ such that for all $n,m \geq N, \rho(x_n,x_m) < \frac{\epsilon}{2K}$. Thus, we have for all $n \geq N$
\begin{equation*}
\begin{split}
\epsilon \leq \rho(x_{n_i}, y') \leq K \max( \rho(x_{n_i}, x_n), \rho(x_n, y') ).
\end{split}
\end{equation*}

Since $K \rho(x_{n_i},x_n) \leq \frac{\epsilon}{2} < \epsilon$, we see that
\[ \frac{\epsilon}{K} \leq \frac{1}{K} \rho(x_{n_i},y') \leq \rho(x_n,y') \]
for $n \geq N$. This implies that the sequence $(x_n)_n$ stays away from $y'$ for large $n$, specifically, $\rho(x_n, y') \geq \frac{\epsilon}{K}$ for $n \geq N$. The same is true for $(x_n)_n$ and $z'$ and some other $\tilde{\epsilon} > 0$. Hence, we have
\begin{equation*}
\begin{split}
\frac{\rho(x_n,x_m)\rho(y',z')}{\rho(x_n,y')\rho(x_m,z')} \leq K^2\frac{\rho(x_n,x_m)\rho(y',z')}{\epsilon \tilde{\epsilon}} \xrightarrow{n,m \rightarrow \infty} 0.
\end{split}
\end{equation*}

We see that $\frac{\rho(x_n,x_m)\rho(y',z')}{\rho(x_n,y')\rho(x_m,z')}$ converges to zero and hence $c(x_n, x_m, y', z') \rightarrow -\infty$.\\

Case 2: Suppose $x_n$ diverges to infinity. We can find a number $N$ such that for all $n \geq N, \rho(x_n, y')\geq \rho(y',z')$ and $\rho(x_n,z') \geq \rho(y',z')$. Then we have
\begin{equation*}
\begin{split}
\frac{\rho(x_n,x_m)\rho(y',z')}{\rho(x_n,y')\rho(x_m,z')} & \leq \frac{K \max(\rho(x_n, y'), \rho(y',x_m)) \rho(y',z')}{\rho(x_n,y')\rho(x_m,z')}\\
& \leq \frac{K^2 \max( \rho(x_n,y'), \rho(y', z'), \rho(z',x_m) ) \rho(y',z')}{\rho(x_n,y')\rho(x_m,z')}\\
& = K^2 \frac{\rho(y',z')}{\min(\rho(x_n,y'),\rho(x_m,z'))} \rightarrow 0.
\end{split}
\end{equation*}

Hence, we see that also in this case, $\frac{\rho(x_n,x_m)\rho(y',z')}{\rho(x_n,y')\rho(x_m,z')}$ converges to zero and, therefore, $c(x_n, x_m, y', z') \rightarrow -\infty$. This completes the proof.

\end{proof}

One might hope that an analogous statement for condition 2 holds. However, the following example illustrates that Lemma \ref{lem:AuxiliaryforCauchy} and \ref{lem:Conditionc} are the best that we can hope for.

\begin{exam} \label{exam:ExampleCauchy}
Consider the circle, represented as $S^1 = \mathbb{R} \diagup 4\mathbb{Z}$. We will mostly use representatives in $[-2, 4]$ to represent points on the circle. Consider the space $Z := S^1 \diagdown \{ [0] \}$ and define a map $\rho : Z \times Z \rightarrow [0,\infty)$ by

\begin{equation*}
\rho([x],[y]) := \begin{cases} \vert x - y \vert & (x,y) \in (0,2]^2 \cup [1,3]^2 \cup [2,4)^2 \cup ([-1,1] \diagdown \{ 0 \})^2\\
2 \vert x - y \vert & (x,y) \in \left( (0,1) \times (2,3) \right) \cup \left( (2,3) \times (0,1) \right) \cup\\
& \hspace{40pt} \cup \left( (1,2) \times (3,4) \right) \cup \left( (3,4) \times (1,2) \right).
\end{cases}
\end{equation*}

Notice the use of different representatives depending on the case. Geometrically, $(Z,\rho)$ can be thought of as follows. Think of $Z$ as a subset of the circle of circumference $4$ with the shortest path metric. This circle can be embedded into $\mathbb{R}^2$ such that it is centered at the origin, i.e. it is the boundary of a disk centered at the origin.

We can now consider the intersection of the circle with each quarter of $\mathbb{R}^2$. We call them the `upper-right', `upper-left', `lower-left', `lower-right' segment of $S^1$, based on their position in the standard coordinate system of $\mathbb{R}^2$.

The distance $\rho(x,y)$ between two points $x,y$ is now defined to be the same as on $S^1$ if $x,y$ lie on the same segment of $S^1$ or if they lie on segments that are neighbours of each other. If $x,y$ lie on segments of $S^1$ that lie opposite to each other, then $\rho(x,y)$ is exactly twice the length of the path from $x$ to $y$ that passes through the point $(0,-1)$.\\

A straightforward computation with several case-distinctions shows that $\rho$ is a $12$-quasi-metric. Thus, we get a  strong M\"obius space $(Z, M_\rho)$. Consider now the following sequence in $Z$:
\[ x_n = \left[\frac{1}{n}(-1)^n\right]. \]

One can show that there is a good pair for $(x_n)_n$ that satisfies condition 3, but not condition 2. Furthermore, one can even find another good pair for $(x_n)_n$ that satisfies both conditions 2 and 3. Specifically, choose $y = 1.5$, $z = -1.5$ for the first case and $y = 1.5$, $z = 1.6$ for the second case.

The issue at hand is that even if we understand the convergence behaviour of $\frac{\rho(x_n,x_m)\rho(y,z)}{\rho(x_n,y)\rho(x_m,z)}$, we cannot control the convergence behaviour of $\frac{\rho(x_n,y)\rho(x_m,z)}{\rho(x_n,z)\rho(x_m,y)}$ if $\rho$ is not a metric. So we have found a quasi-metric -- and thus a  strong M\"obius structure $M_\rho$ -- for which the statement "3) $\Rightarrow$ 2)" that we have proven for metrics in Lemma \ref{lem:CauchySequences}, does not hold.\\

This example illustrates the relationship between the different possible conditions one could use to define Cauchy sequences in a  strong M\"obius space. If condition 2 holds for one good pair $y, z$, this does not imply that condition 2 holds for all good pairs, unless we work with a metric space. In the same way, if condition 3 holds for all good pairs, this doesn't imply the same for condition 2. However, from Lemma \ref{lem:Conditionc} we know that, if condition 3 holds for one good pair, it holds for all of them. 

\end{exam}

Example \ref{exam:ExampleCauchy} leads us to the following definition of Cauchy sequences in a  strong M\"obius space.

\begin{mydef} \label{def:CauchySequence}
Let $(Z,M)$ be a  strong M\"obius space. A sequence $(x_n)_n$ in $Z$ is called a {\it Cauchy sequence} if and only if for one (and hence all) good pairs $y,z$ in $Z$, we have
\[ c(x_n,x_m,y,z) \xrightarrow{n,m \rightarrow \infty} -\infty. \]
\end{mydef}

\begin{mydef} \label{def:Completeness}
A  strong M\"obius space $(Z,M)$ is called {\it complete} if and only if all Cauchy sequences in $(Z,M)$ converge.
\end{mydef}

Using the previous lemmata, the following results are easy to see.

\begin{prop} \label{prop:MobiusInvarianceofCauchySequencesandCompleteness}

Let $(Z,M), (Z',M')$ be two  strong M\"obius spaces, $f : Z \rightarrow Z'$ a M\"obius equivalence between them.

\begin{enumerate}

\item[1)] Let $(x_n)_n$ be a sequence in $Z$. Then $(x_n)_n$ is a Cauchy sequence in $(Z,M)$ if and only if $(f(x_n))_n$ is a Cauchy sequence in $(Z',M')$.

\item[2)] The  strong M\"obius space $(Z,M)$ is complete if and only if $(Z',M')$ is.

\end{enumerate}
\end{prop}

\begin{proof}
Proof of 1): The sequence $(x_n)_n$ is a Cauchy sequence if and only if for some good pair $y,z$ in $Z$, we have

\[ c(x_n,x_m,y,z) \rightarrow -\infty. \]

Since $f$ is a M\"obius equivalence, this implies

\[ c'(f(x_n),f(x_m), f(y),f(z)) = c(x_n,x_m,y,z) \rightarrow -\infty. \]

Since $f$ is a homeomorphism by Lemma \ref{lem:homeomorphism} and $y,z$ is a good pair, so is $f(y),f(z)$ for $(f(x_n))_n$. Thus, $(f(x_n))_n$ is a Cauchy sequence in $(Z',M')$.\\

Proof of 2): Suppose $(Z,M)$ is complete and let $(x'_n)_n$ be a Cauchy sequence in $(Z',M')$. By the first part of the Proposition, $(f^{-1}(x'_n))_n$ is a Cauchy sequence in $(Z,M)$ which converges to some $x \in Z$ by completeness. Since $f$ is a homeomorphism, $(x'_n)_n$ has to converge to $f(x)$. This implies completeness.

\end{proof}

The notion of completeness defined above compares to the notion of completeness defined in metric spaces as follows:

\begin{thm}\label{thm:MetricvsMobiusCompletion}
Let $(Z,\rho)$ be a (possibly extended) metric space and denote the induced M\"obius structure by $M$. Equivalent are:

\begin{enumerate}

\item[1)] $(Z,M)$ is complete as a  strong M\"obius space

\item[2)] $(Z,\rho)$ is complete as a metric space and is either bounded or has a point at infinity.

\end{enumerate}
\end{thm}

\begin{proof}
1) $\Rightarrow$ 2): Suppose, $(Z,M)$ is complete as a  strong M\"obius space and let $(x_n)_n$ be a Cauchy sequence in the metric sense. By Lemma \ref{lem:CauchySequences}, $(x_n)_n$ is also a Cauchy sequence in the M\"obius sense. Hence, $(x_n)$ has to converge in M\"obius topology. Since the M\"obius topology is the same as the metric topology on a metric space by Theorem \ref{thm:metrictopology}, $(x_n)_n$ converges in metric topology and $(Z,\rho)$ is complete in the metric sense.

2) $\Rightarrow$ 1): Suppose $(Z,\rho)$ is complete as a metric space and let $(x_n)_n$ be a Cauchy sequence in the M\"obius sense. By Lemma \ref{lem:CauchySequences}, $(x_n)_n$ is either a Cauchy sequence in the metric sense, or it diverges to infinity. If it is a Cauchy sequence in the metric sense, it converges in metric topology (and thus in M\"obius topology) by metric completeness. If $x_n$ diverges to infinity, the metric space cannot be bounded. Hence, it has a point at infinity by assumption and $x_n$ converges to the point at infinity in metric and M\"obius topology.
\end{proof}




\section{Constructing the completion} \label{sec:Completion}

Now that we have a notion of Cauchy sequences and a notion of completeness for  strong M\"obius spaces, an obvious question is whether every strong M\"obius space has a naturally unique completion, as metric spaces do.

Certainly, if we take a metric space $(Z,\rho)$ and consider the induced  M\"obius structure $M$, the metric completion $(\overline{Z},\overline{\rho})$ is either complete with respect to the induced  M\"obius structure $\overline{M}$, which is just an extension of $M$, or one has to add one point at infinity to make it complete in the M\"obius sense. Adding a point at infinity doesn't change that $Z$ is dense in its completion and it is easy to see that uniqueness up to isometry for the metric case implies uniqueness up to M\"obius equivalence (even up to isometry) in the M\"obius sense.\\

We want to see whether we can create a completion even beyond the metric case. For this, we will do the same construction that is used to construct the metric completion. Let $(Z,crt)$ be a  strong M\"obius space. (We will talk about the necessary extra condition later.) Define the set
\[ \overline{Z} := \{ (x_n)_n \vert (x_n) \text{ a Cauchy sequence in } (Z,crt)  \} \diagup \sim \]

where $(x_n) \sim (x'_n)$ if and only if for every pair $y \neq z$ in $Z$ that is a good pair for both $(x_n)$ and $(x'_n)$, we have
\[ c(x_n,x'_n,y,z) \rightarrow -\infty. \]

There is a canonical embedding of $Z$ into $\overline{Z}$ by sending $x$ to the constant sequence $x_n = x$. This is clearly a Cauchy sequence and the map $x \mapsto [(x)_n]$ is injective, since two different constant sequences are not equivalent in the sense defined above.\\

The next step is to extend the  M\"obius structure $crt$ to $\overline{Z}$. We would like to define
\[ \overline{crt}([(w_n)],[(x_n)],[(y_n)],[(z_n)]) := \lim_{n \rightarrow \infty} crt(w_n,x_n,y_n,z_n). \]

There are two questions that arise immediately when stating this definition. Does the limit on the right-hand-side exist and is it independent of the choice of representative of a point $[(w_n)] \in \overline{Z}$? In general, the answer to these two questions is no. The reason for that has already appeared in example \ref{exam:ExampleCauchy}, namely that, if $\rho(x_n,x_m) \rightarrow 0$, we cannot make sure that $\rho(x_n,y)$ converges for all $y \in Z$. Specifically, the sequence $x_n$ discussed in example \ref{exam:ExampleCauchy} satisfies $crt(x_{2n},x_{2n+1},y,z) \rightarrow (0:1:4)$ and $crt(x_{2n}, x_{2n+2},y,z) \rightarrow (0:1:1)$. Therefore $\lim \limits_{n,m \rightarrow \infty}crt(x_n,x_m,y,z)$ does not exist. This example is a special case that will appear in the definition of $\overline{crt}$ given above and makes this construction not well-defined in general.

As mentioned in example \ref{exam:ExampleCauchy}, the problem at hand is that we cannot control the behaviour of $\frac{\rho(x_n,y)\rho(x_m,z)}{\rho(x_m,y)\rho(x_n,z)}$ for a Cauchy sequence $(x_n)$. If we knew that $crt(w_n,x_n,y_n,z_n)$ could only converge to points in $\mathbb{R}P^2$ that are allowed to be obtained by a  M\"obius structure, then we could resolve this problem (as we will see below). The following property makes sure that these issues cannot arise.

\begin{mydef} \label{mydef:symmetrycondition}
A M\"obius structure $crt$ or a  M\"obius space $(Z,crt)$ satisfies the (symmetry)-condition if and only if
\[ \overline{\Ima(crt)} \subseteq \overline{\Delta} = \{ (a:b:c) \vert a,b,c > 0 \} \cup \{ (0:1:1), (1:0:1), (1:1:0) \} \]
where $\overline{\Ima(crt)}$ denotes the closure of the image of $crt$ in $\mathbb{R}P^2$.
\end{mydef}

\begin{figure} 
\begin{tikzpicture}
\draw [dashed] (6,0) -- (9,0);
\draw [dashed] (6,0) -- (7.5, 2.6);
\draw [dashed] (9,0) -- (7.5, 2.6);
\draw [fill] (6.75,1.3) circle [radius=0.05];
\draw [fill] (8.25,1.3) circle [radius=0.05];
\draw [fill] (7.5,0) circle [radius=0.05];
\draw [thick] (6.75,1.3) to [out=240, in=200] (6.8,0.7) to [out=20, in=120] (7.125,0.65) to [out=300, in=30] (7.01,0.343) to [out=220, in=180] (7.5,0) to [out=0, in=320] (7.99,0.343) to [out=140, in=240] (7.875, 0.65) to [out=60, in=160] (8.2,0.7) to [out=340, in=300] (8.25,1.3) to [out=120, in=80] (7.706, 1.555) to [out=260, in=0] (7.5, 1.299) to [out=180, in=280] (7.292, 1.552) to [out=100, in=60] (6.75,1.3);
\node [below left] at (5.7,-0.3) {$(1:0:0)$};
\node [below right] at (9.3,-0.3) {$(0:1:0)$};
\node [above] at (7.5,3) {$(0:0:1)$};
\node at (7.5,1.0) {$\Ima(crt)$};
\end{tikzpicture}
\caption{A  M\"obius structure $crt$ satisfies the (symmetry)-condition if and only if no point in the boundary of $\overline{\Delta}$ can be approximated by a sequence of points in $\Ima(\Delta)$ except for $(\frac{1}{2} : \frac{1}{2} : 0), (\frac{1}{2} : 0 : \frac{1}{2})$ and $(0 : \frac{1}{2} : \frac{1}{2})$. In other words, the image doesn't touch the boundary at any other than those three points.}
\end{figure}
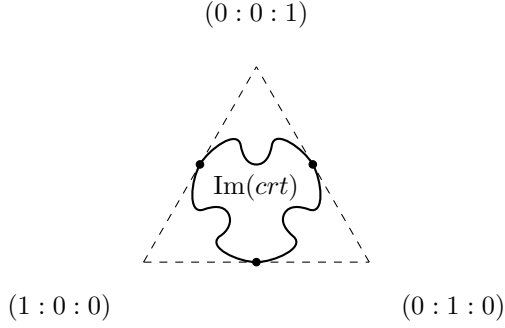

To interpret this definition, it is useful to think of $\Delta \subset \mathbb{R}P^2$ as a triangle. Specifically, consider the triangle $\{ (x, y, z) \in \mathbb{R} \vert x + y + z = 1, x, y, z \geq 0 \}$. The projection of this triangle onto $\mathbb{R}P^2$ is exactly the topological closure of $\Delta$. The (symmetry)-condition tells us that any sequence of cross-ratio-triples $crt(w_n,x_n,y_n,z_n)$ can only accumulate at points in the interior of this triangle or at one of the three distinct points on the boundary of the triangle that are assumed by degenerate quadruples. It turns out that this is the property needed to construct a completion.

\begin{thm} \label{thm:Completion}
Let $(Z, M)$ be a M\"obius space that satisfies the (symmetry)-condition. Then there exists a complete strong M\"obius space $(\overline{Z}, \overline{crt})$ with a M\"obius embedding $i_Z : Z \hookrightarrow \overline{Z}$ (i.\,e.\,$\overline{crt}(i_Z(w), i_Z(x), i_Z(y), i_Z(z)) = crt(w, x, y, z)$ for all admissible quadruples $(w, x, y, z)$) such that $i_Z(Z)$ is dense in $\overline{Z}$.

Furthermore, if $(Z', crt')$ is a complete strong M\"obius space, such that there exists a M\"obius embedding $i : Z \hookrightarrow Z'$ such that $i(Z)$ is dense in $Z'$, then there exists a unique M\"obius equivalence $f : \overline{Z} \rightarrow Z'$ such that $i = f \circ i_Z$.

\end{thm}

The space $(\overline{Z}, \overline{crt})$ is going to be the one constructed above. Suppose $(Z,crt)$ satisfies the (symmetry)-condition. Let $\rho$ be a quasi-metric inducing $crt$, $(x_n)$ a Cauchy sequence in the M\"obius sense and $y, z$ a good pair for $(x_n)$. By symmetry of $x_n,x_m$ we see that $\frac{\rho(x_n,x_m)\rho(y,z)}{\rho(x_n,y)\rho(x_m,z)}$ and $\frac{\rho(x_n,x_m)\rho(y,z)}{\rho(x_m,y)\rho(x_n,z)}$ both converge to zero as $n,m$ tend to infinity. Therefore, the sequence $crt(x_n,x_m,y,z)$ can be written in the form $(a_n : b_n : c_n)$ with all three entries being non-negative, where we scale $a_n,b_n,c_n$ such that $a_n+b_n+c_n = 2$. By the convergence statements above, $a_n$ has to converge to zero. Since $crt$ satisfies the (symmetry)-condition, the only point $(0:b:c)$ that can be approximated arbitrarily well in $\Ima(crt)$ is $(0:1:1)$. Therefore, $crt(x_n,x_m,y,z) \xrightarrow{n, m \rightarrow \infty} (0:1:1)$.

The (symmetry)-condition allows us to prove a result about convergence that will be useful in proving Theorem \ref{thm:Completion}.

\begin{prop} \label{prop:Convergence}
Let $(Z,crt)$ be a  strong M\"obius structure satisfying the (symmetry)-condition. Let $(x_n)$, $(y_n)$ be Cauchy sequences in $Z$, $y \in Z$ and $\rho$ a quasi-metric that induces $crt$ and has a point at infinity (e.g. $\rho = \rho_A$). Then $\rho(x_n,y)$ and $\rho(x_n, y_n)$ converge, possibly to infinity.
\end{prop}

Recall that every sequence $(x_{n,m})$ in $\mathbb{R}$ parametrized by $\mathbb{N}^2$ with the property that $\lim \limits_{n \rightarrow \infty} x_{n,m}$ exists for every $m$, $\lim \limits_{m \rightarrow \infty} x_{n,m}$ exists for every $n$ and $\lim \limits_{n,m \rightarrow \infty} x_{n,m}$ exists, satisfies
\[ \lim_{n \rightarrow \infty} \lim_{m \rightarrow \infty} x_{n,m} = \lim_{m \rightarrow \infty} \lim_{n \rightarrow \infty} x_{n,m} = \lim_{n,m \rightarrow \infty} x_{n,m}. \]

\begin{proof}
Denote the point at infinity with respect to $\rho$ by $\infty$. By Lemma \ref{lem:AuxiliaryforCauchy}, $x_n$ is either bounded or diverges to infinity. If $x_n$ diverges to infinity with respect to $\rho$, then $\rho(x_n, y) \rightarrow \infty$. Now assume the Cauchy sequence $x_n$ is bounded with respect to $\rho$. By Lemma \ref{lem:AuxiliaryforCauchy}, we know that $\rho(x_n, x_m) \xrightarrow{n, m \rightarrow \infty} 0$. In particular, since $\rho$ is a quasi-metric, either $\rho(x_n, y) \xrightarrow{n \rightarrow \infty} 0$, or there exists $\epsilon > 0$, such that $\rho(x_n, y) \geq \epsilon$ for all $n$ sufficiently large. Suppose $\rho(x_n, y)$ does not converge to zero. Then, $y, \infty$ are a good pair for $(x_n)$, $c(x_n, x_m, y, \infty) \xrightarrow{n, m \rightarrow \infty} -\infty$ and by the (symmetry)-condition
\[ crt(x_n,x_m,y,\infty) \xrightarrow{n,m \rightarrow \infty} (0:1:1). \]

This implies that
\begin{equation} \label{eq:ConvergenceCauchy}
\frac{\rho(x_n,y)}{\rho(x_m,y)} \xrightarrow{n,m \rightarrow \infty} 1.
\end{equation}

We can now use this to prove that $\rho(x_n,y)$ converges for every Cauchy sequence $(x_n)$ and any $y \in Z$. If $(x_n)$ converges to $y$, then $\rho(x_n,y) \rightarrow 0$ by definition. If $(x_n)$ diverges to infinity with respect to $\rho$, then $\rho(x_n,y) \rightarrow \infty$. If $(x_n)$ is bounded with respect to $\rho$ then $0 \leq \rho(x_n,y) \leq B$ and hence - by compactness - has a convergent subsequence $\rho(x_{n_i},y)$. Applying equation (\ref{eq:ConvergenceCauchy}) in the case $m = n_i$ yields
\[ \frac{\rho(x_n,y)}{\rho(x_{n_i},y)} \xrightarrow{n,i \rightarrow \infty} 1. \] 
Since $\rho(x_{n_i},y)$ converges, this implies that the limit of $\rho(x_n,y)$ exists and
\[ \lim_{n \rightarrow \infty} \rho(x_n,y) = \lim_{i \rightarrow \infty} \rho(x_{n_i},y). \]

Now consider the two Cauchy-sequences $(x_n)$, $(y_n)$. If one of the sequences is bounded and the other diverges to infinity, then $\rho(x_n, y_n) \rightarrow \infty$. If both sequences diverge to infinity, replace $\rho$ with an involution $\rho_o$ at any point $o \in Z$. Both $(x_n)$ and $(y_n)$ are bounded with respect to $\rho_o$. Convergence of $\rho_o(x_n, y_n)$ and the fact that $\rho$ is the involution of $\rho_o$ at the point $\infty \in Z$ will imply convergence of $\rho(x_n, y_n)$.

We are left to prove convergence of $\rho(x_n, y_n)$ when both sequences are bounded. In this situation, we know that $\rho(x_n, x_m), \rho(y_n, y_m) \xrightarrow{n, m \rightarrow \infty} 0$. Suppose $\rho(x_n, y_n)$ does not converge to zero. Then the limits above and the fact that $\rho$ is a quasi-metric imply that there exists some $\epsilon > 0$, such that for all $n$ sufficiently large, $\rho(x_n, y_n) > \epsilon$. We conclude that
\[ cr(x_n, x_m, y_n, \infty) = \frac{ \rho(x_n, x_m) }{ \rho(x_n, y_n) } \xrightarrow{n, m \rightarrow \infty} 0. \]
Since $crt$ satisfies the (symmetry)-condition, this implies that
\begin{equation} \label{eq:LimitEquation} \frac{ \rho(y_n, x_n) }{ \rho(y_n, x_m) } = cr(y_n, x_n, x_m, \infty) \xrightarrow{n, m \rightarrow \infty}Ê1. \end{equation}
Furthermore, replacing either $n$ or $m$ by a subsequence does not change this convergence behaviour. The same argument with the roles of $(x_n)$ and $(y_n)$ swapped implies
\[ \frac{ \rho(x_n, y_n) }{ \rho(x_n, y_m) } \xrightarrow{n, m \rightarrow \infty} 1.Ê\]
Since $(x_n)$, $(y_n)$ are bounded, there exist subsequences $(x_{n_i})$, $(y_{n_i})$ such that $\rho(x_{n_i}, y_{n_i})$ converges. Equation \ref{eq:LimitEquation} now implies that
\[ \lim_{i, m \rightarrow \infty} \frac{ \rho(y_{n_i}, x_{n_i}) }{ \rho(y_{n_i}, x_m) } = 1 \]
and therefore,
\[ \lim_{i \rightarrow \infty} \rho(y_{n_i}, x_{n_i}) = \lim_{i, n \rightarrow \infty} \rho(y_{n_i}, x_n).Ê\]
Using equation \ref{eq:LimitEquation} with the roles of $(x_n)$, $(y_n)$ swapped, we obtain
\[ \lim_{i, n \rightarrow \infty} \frac{ \rho(x_n, y_n) }{ \rho(x_n, y_{n_i}) } = 1 \]
and therefore,
\[ \lim_{n \rightarrow \infty} \rho(x_n, y_n) = \lim_{i, n \rightarrow \infty} \rho(x_n, y_{n_i}). \]
This implies that $\rho(x_n, y_n)$ converges whenever both sequences are Cauchy-sequences (provided that $\rho$ has a point at infinity).
\end{proof}

\begin{proof}[Proof of Theorem \ref{thm:Completion}]

Let $\overline{Z}$ and $\overline{crt}$ as defined before. We start by proving that $\overline{crt}$ is well-defined. Let $(w_n)$, $(x_n)$, $(y_n)$, $(z_n)$ be Cauchy-sequences in $Z$. By Proposition \ref{prop:Convergence}, $\rho(\cdot_n, \cdot_n)$ converges for any two of the sequences. Therefore, $crt(w_n, x_n, y_n, z_n)$ converges as well and, by the (symmetry)-condition, it converges to a point in $\overline{\Ima(crt)} \subseteq \overline{\Delta}$.

We are left to show that $\lim \limits_{n \rightarrow \infty} crt(w_n,x_n,y_n,z_n) = \lim \limits_{n \rightarrow \infty} crt(w'_n,x'_n,y'_n,z'_n)$ for $(w_n) \sim (w'_n), (x_n) \sim (x'_n), (y_n) \sim (y'_n), (z_n) \sim (z'_n)$. Again, we will prove the statement for $\rho(x_n,y)$ and a quasi-metric $\rho$ that induces $crt$ and has a point at infinity. Repeating this argument then implies as above that the statement for $crt(w_n,x_n,y_n,z_n)$ holds.

So let $\rho$ be a quasi-metric that induces $crt$ and has a point at infinity, denoted by $\infty$. Let $(x_n) \sim (x'_n)$. Since $c(x_n,x'_n,y,z) \rightarrow -\infty$ for all good pairs, it is easy to see that either $\rho(x_n,x'_n) \rightarrow 0$ or $x_n, x'_n$ both diverge to infinity.

Let $y \in Z$. If $(x_n)$ diverges to $\infty$, then $x'_n$ has to diverge to infinity as well and hence $\rho(x_n,y) = \rho(x'_n,y)$ for all $y \in Z$.

Now suppose, $(x_n)$ doesn't diverge to $\infty$, hence it has to be bounded by Lemma \ref{lem:AuxiliaryforCauchy} and $\rho(x_n, x'_n) \xrightarrow{n \rightarrow \infty} 0$. By Proposition \ref{prop:Convergence}, $\rho(x_n,y)$ and $\rho(x'_n,y)$ both converge. Suppose $\rho(x_n,y) \xrightarrow{n \rightarrow \infty} 0$. Then
\[ \rho(x'_n,y) \leq K \max(\rho(x'_n,x_n), \rho(x_n,y)) \xrightarrow{n \rightarrow \infty} 0. \]
Thus, $\lim \limits_{n \rightarrow \infty} \rho(x'_n,y) = 0 = \lim \limits_{n \rightarrow \infty} \rho(x_n,y)$.

Finally, suppose, $\rho(x_n, y) \rightarrow r$ for some positive real number. Then, by swapping $x_n$ and $x'_n$ in the argument above, $\rho(x'_n,y)$ doesn't converge to zero. Therefore and because $(x_n),(x'_n)$ are both bounded, $y, \infty$ is a good pair for both sequences. Since the two sequences are equivalent by assumption, we have
\[ c(x_n,x'_n,y, \infty) \rightarrow -\infty. \]
The (symmetry)-condition implies
\[ crt(x_n,x'_n,y,\infty) \rightarrow (0:1:1). \]
In other words,
\[ crt(x_n,x'_n,y,\infty) = \frac{\rho(x_n,y)}{\rho(x'_n,y)} \rightarrow 1 \]
and therefore
\[ \lim_{n \rightarrow \infty} \rho(x_n,y) = \lim_{n \rightarrow \infty} \rho(x'_n,y). \]

Analogously to the second half of the proof of Proposition \ref{prop:Convergence}, we show that $\lim_{n \rightarrow \infty} \rho(x_n,y_n) = \lim_{nÊ\rightarrow \infty} \rho(x'_n,y_n)$ for all Cauchy sequences $(x_n) \sim (x'_n), (y_n)$. Thus, $\lim_{n \rightarrow \infty} crt(w_n,x_n,y_n,z_n) = \lim_{n \rightarrow \infty} crt(w'_n,x'_n,y'_n,z'_n)$ and therefore, $\overline{crt}$ is well-defined.\\

Given a  M\"obius space $(Z, crt)$ that satisfies the (symmetry)-condition, we have constructed a new  strong M\"obius space $(\overline{Z}, \overline{crt})$. We also have a canonical map of $Z$ into $\overline{Z}$ that preserves the  M\"obius structure (hence it is also a topological embedding).\\

We are left to show that $\overline{Z}$ is complete and that $\overline{Z}$ is unique. We prove completeness first. Let $\xi^m = [(x^{(m)}_n)_n] \in \overline{Z}$ be such that $(\xi^m)_m$ is a Cauchy sequence in $\overline{Z}$. We will often identify $\xi^m$ with the representative $(x^{(m)}_n)$. Choose a quasi-metric $\rho$ on $Z$ that induces $crt$ and let $\overline{\rho}$ be the extension to $\overline{Z}$. Clearly, $\overline{\rho}$ induces $\overline{crt}$. By Lemma \ref{lem:AuxiliaryforCauchy}, $(\xi^m)_m$ either diverges to infinity, or it is bounded with respect to $\overline{\rho}$.

We analyze the point at infinity in $\overline{Z}$ with respect to $\overline{\rho}$. Let it be represented by a Cauchy sequence $(z_n)$ in $Z$. Then $\overline{\rho}((z_n),(y_n)) = \infty$ for all Cauchy sequences $(y_n)$ in $Z$ that are not equivalent to $(z_n)$. This means that

\[ \infty = \overline{\rho}((z_n),(y_n)) = \lim_{n \rightarrow \infty} \rho(z_n,y_n),Ê\]

which is the same as saying that $(z_n)$ diverges to infinity. So the point at infinity with respect to $\overline{\rho}$ is the equivalence class of all sequences in $Z$ that diverge to infinity with respect to $\rho$.\\

Before we study the convergence of our sequence $(\xi^m)_m$, we need to take a look at convergence in the M\"obius topology. Given a strong M\"obius space $(Z', M')$, a sequence $x_n$ in $Z'$ converges to $x$ if and only if for all $\rho_A$ and all $y \in Z'$ such that $y$ does not lie at infinity with respect to $\rho_A$, $\rho_A(x_n, y) \rightarrow \rho_A(x,y)$. By Lemma \ref{lem:cornerconditionatinfinity} and Lemma \ref{lem:InvolutionFormularho_A}, if a M\"obius structure $crt$ is induced by a quasi-metric $\rho$, then the induced semi-metrics $\rho_A$ are quasi-metrics and have the form
\[ \rho_A(x,y) = \frac{\rho(x,y)}{\rho(x, \omega)\rho(\omega,y)} \frac{\rho(\alpha, \omega)\rho(\omega, \beta)}{\rho(\alpha, \beta)}. \]

We see that, as long as $x_n$ does not diverge to infinity with respect to $\rho$, it is sufficient to prove that $\rho(x_n, y) \rightarrow \rho(x,y)$ for all $y$. In particular, since every strong M\"obius structure is induced by a bounded quasi-metric $\rho$ by Proposition \ref{prop:MobiusStructuresQuasiMetrics}, we can simply use such a quasi-metric to study convergence.

Returning to the space $(\overline{Z}, \overline{crt})$ constructed above, if we pick a bounded quasi-metric $\rho$ that induces $crt$, then $\overline{\rho}$ will be a bounded quasi-metric as well. The discussion above implies that a sequence $(\xi^m)_m$ converges to a point $\xi$ if and only if $\overline{\rho}(\xi^m, \eta) \rightarrow \overline{\rho}(\xi,\eta)$ for all $\eta \in \overline{Z}$.\\

Back to the sequence $(\xi^m)_m$. Since we assume $\rho$ to be bounded, any Cauchy-sequence in $(Z, M)$ is bounded with respect to $\rho$. We need to find a Cauchy sequence $(x_l)_l$ in $Z$ such that $(\xi^m)_m$ converges to that sequence in M\"obius topology as $m$ tends to infinity. Since $\rho$ is bounded, $\xi^m = [ (x^{(m)}_n)_n ]$ can be represented by a bounded Cauchy sequence for every $m$. By Lemma \ref{lem:AuxiliaryforCauchy},
\[ \rho(x^{(m)}_n, x^{(m)}_{n'}) \xrightarrow{n,n' \rightarrow \infty} 0. \]
Thus, for every fixed $m$ and every $\epsilon > 0$, we find a natural number $N_m$, such that for all $n, n' \geq N_m$, we have
\begin{equation*} \label{eq:estimatealongn}
\rho(x^{(m)}_n, x^{(m)}_{n'}) < \epsilon.
\end{equation*}
Let $(y_n)$ be a Cauchy sequence in $Z$. Since $\overline{\rho}$ is bounded, the sequence $(\xi^m)_m$ is bounded and we find some constant $B > 0$ such that $\overline{\rho}(\xi^m, (y_n)) < B$ for all $m \in \mathbb{N}$. Therefore, for every $m$ we find some natural number $\overline{N}_m$ such that for all $n \geq \overline{N}_m$, we have
\[\rho(x^{(m)}_n, y_n) \leq 2B.\]
Since $(\xi^m)_m$ is a bounded Cauchy sequence by assumption, we also find for every $\epsilon > 0$ a natural number $M$ such that for all $m, m' \geq M$,
\begin{equation*} \label{eq:estimatealongm}
\overline{\rho}(\xi^m, \xi^{m'}) < \epsilon.
\end{equation*}

We now use the following technical lemma.

\begin{lem}\label{lem:technicallemma}

There exists a sequence $(x_l)_l$ in $Z$, satisfying the following properties:

\begin{enumerate}

\item $x_l = x^{(m_l)}_{n_l}$

\item The sequences $m_l, n_l$ are increasing.

\item For every $l \in \mathbb{N}$ and all $n \geq n_l$, $\rho(x^{(m_l)}_{n_l}, x^{(m_l)}_n) < \frac{1}{lK}$.

\item For every $l \in \mathbb{N}$ and all $m,m' \geq m_l, \overline{\rho}( \xi^m, \xi^{m'} ) \leq \frac{1}{2lK}$.

\item For all $l \leq l' \in \mathbb{N}$ and all $n \geq n_{l'}$, $\rho(x^{(m_l)}_{n}, x^{(m_{l'})}_{n}) < \frac{1}{lK}$.

\end{enumerate}

\end{lem}

We first show how the Lemma completes the proof of Theorem \ref{thm:Completion}. Given such a sequence $(x_l)_l$, one immediately sees that for all $l$ and all $l' \geq l$, we have

\[ \rho(x_l, x_{l'}) = \rho(x^{(m_l)}_{n_l}, x^{(m_{l'})}_{n_{l'}}) \leq K \max ( \rho( x^{(m_l)}_{n_l}, x^{(m_l)}_{n_{l'}} ), \rho( x^{(m_l)}_{n_{l'}}, x^{(m_{l'})}_{n_{l'}} ) ) \leq K \frac{1}{lK} = \frac{1}{l}. \]
This implies that $x_l$ is bounded and a Cauchy sequence. Furthermore, for $m \geq m_{l_0}$
\begin{equation*}
\begin{split}
\overline{\rho}(\xi^m, (x_l)_l) & = \lim_{l \rightarrow \infty} \rho( x^{(m)}_l, x^{(m_l)}_{n_l})\\
& \leq \lim_{l \rightarrow \infty} K^3 \max \Bigl( \rho \left( x^{(m)}_l, x^{(m_{l_0})}_l \right), \rho \left( x^{(m_{l_0})}_l, x^{(m_{l_0})}_{n_{l_0}} \right),\\
& \hspace{78pt} \rho \left( x^{(m_{l_0})}_{n_{l_0}}, x^{(m_{l_0})}_{n_l} \right), \rho \left( x^{(m_{l_0})}_{n_l}, x^{(m_l)}_{n_l} \right) \Bigr) \\
& \leq \lim_{l \rightarrow \infty} \frac{K^2}{ l_0} =  \frac{K^2}{l_0}.
\end{split}
\end{equation*}

Therefore, $\overline{\rho}(\xi^m, (x_l)_l) \xrightarrow{m \rightarrow \infty} 0$ and for any other point $(y_l)_l \in \overline{Z}$, we find $\epsilon_y$ such that $\overline{\rho}(\xi^m, (y_l)_l) > \epsilon_y$ for $m$ sufficiently large. This implies that for all $(y_l)_l, (z_l)_l \in \overline{Z} \diagdown \{ (x_l)_l \}$, we have $\overline{crt}(\xi^m, (x_l)_l, (y_l)_l, (z_l)_l) \xrightarrow{m \rightarrow \infty} (0:1:1)$. Since we assume that $(\xi^m)_m$ does not diverge to infinity, $(x_l)_l \neq \infty$ and we can choose $(z_l)_l = \infty$ (by having chosen the original $\rho$ to have a point at infinity). Then, denoting $y := (y_l)_l, \infty = (\infty)_l$, this limit takes the form
\[ \overline{crt}(\xi^m, (x_l)_l, y, \infty) \xrightarrow{m \rightarrow \infty} (0:1:1). \]
By the definition of $\overline{crt}$ this implies
\[ \frac{\overline{\rho}(\xi^m, y)}{\overline{\rho}((x_l)_l,y)} \xrightarrow{m \rightarrow \infty} 1. \]
In other words,
\[ \lim_{m \rightarrow \infty} \overline{\rho}(\xi^m, y) = \overline{\rho}((x_l)_l, y). \]

This implies that $\xi^m$ converges to $(x_l)$.\\

We are left to prove the technical Lemma and to show that the completion $(\overline{Z}, \overline{crt})$ is unique up to unique M\"obius equivalence. Let $(Z', crt')$ be a complete  strong M\"obius space and $i : Z \hookrightarrow Z'$ a M\"obius embedding, i.e. an injective map that is a M\"obius equivalence onto its image. Further, assume $i(Z)$ is dense in $Z'$ with its M\"obius topology. Denote the canonical inclusion of $Z$ into $\overline{Z}$ by $i_Z$. Since $i, i_Z$ are both injective, we get a bijection $f : i(Z) \rightarrow i_Z(Z)$ which sends $i(x)$ to $i_Z(x)$. Since $i, i_Z$ are M\"obius equivalences onto their images, they are also homeomorphisms onto their images. Therefore, the map $f$ is a homeomorphism with respect to the subspace topology on $i(Z)$ and $i_Z(Z)$. Since $f$ preserves the M\"obius structure and therefore Cauchy-sequences and equivalent Cauchy-sequences, it extends to a bijection $F : Z' \rightarrow \overline{Z}$.

We claim that $F$ is a M\"obius equivalence. Let $(w,x,y,z)$ be a non-degenerate quadruple in $Z'$ (clearly, $F$ preserves the  M\"obius structure on degenerate, admissible quadruples). Then we can approximate these four points by sequences $w_n, x_n, y_n,z_n$ in $i(Z)$. By definition of $F$,
\[ F(w) = \lim_{n \rightarrow \infty} F(w_n) \]
\[ F(x) = \lim_{n \rightarrow \infty} F(x_n) \]
\[ F(y) = \lim_{n \rightarrow \infty} F(y_n) \]
\[ F(z) = \lim_{n \rightarrow \infty} F(z_n) \]
and hence
\begin{equation*}
\begin{split}
\overline{crt}(F(w)F(x)F(y)F(z)) & = \lim_{n \rightarrow \infty} crt(F(w_n)F(x_n)F(y_n)F(z_n))\\
& = \lim_{n \rightarrow \infty} crt(f(w_n)f(x_n)f(y_n)f(z_n))\\
& = \lim_{n \rightarrow \infty} crt'(w_n,x_n,y_n,z_n)\\
& = crt'(w,x,y,z). 
\end{split}
\end{equation*}

This shows that $F$ preserves the  M\"obius structure on non-degenerate quadruples. Hence, $F$ is a M\"obius equivalence. Since all M\"obius equivalences are homeomorphisms, uniqueness follows from the fact that $F \vert_{i(Z)} = f$ is given and the fact that $i(Z)$ is dense in $Z'$. This completes the proof of Theorem \ref{thm:Completion} up to the proof of Lemma \ref{lem:technicallemma}.

\end{proof}

\begin{proof}[Proof of Lemma \ref{lem:technicallemma}]

We are left to construct the sequence $x_l$. We construct $x_l$ inductively. The induction start goes as follows: Since $(\xi^m)_m$ is a bounded Cauchy-sequence, we find natural numbers $M_1 < M_2$, such that
\[ \forall m,m' \geq M_1,\, \overline{\rho}(\xi^m, \xi^{m'}) < \frac{1}{2K} \]
\[ \forall m,m' \geq M_2,\, \overline{\rho}(\xi^m, \xi^{m'}) < \frac{1}{4K}. \]
Now we fix $m = M_1, m' = M_2$. We find a natural number $N_1$ such that
\[ \forall n \geq N_1,\, \rho\left(x^{(M_1)}_n, x^{(M_2)}_n\right) < \frac{1}{K}. \]
Since $(x^{(M_1)}_n)_n$ is a bounded Cauchy sequence in $Z$, we can choose $N_1$ such that additionally,
\[ \forall n,n' \geq N_1,\, \rho\left(x^{(M_1)}_n, x^{(M_1)}_{n'}\right) < \frac{1}{K}. \]
Set
\[ x_1 := x^{(M_1)}_{N_1}. \]
We see that $x_1$ satisfies conditions 3 and 4 from above. Now we do the inductive construction.

Suppose, we are given points $x_1, \dots, x_l$ in $Z$ satisfying properties 1-5. Since $(\xi^m)_m$ is a Cauchy sequence in $\overline{Z}$, we find some $M_{l+1} > m_l$, such that

\[ \forall m,m' \geq M_{l+1}, \overline{\rho}( \xi^m, \xi^{m'}) < \frac{1}{2(l+1)K}. \]

Put $m_{l+1} := M_{l+1}$. Since we have chosen $M_{l+1} > m_l$, condition 2 stays satisfied for $(m_l)_l$. Furthermore, $m_{l+1}$ satisfies condition 4. Since $\xi^{m_{l+1}}$ is a Cauchy sequence, we find some natural number $N_0$, such that

\[ \forall n,n' \geq N_0, \rho(x^{(m_{l+1})}_n, x^{(m_{l+1})}_{n'}) < \frac{1}{(l+1)K}. \]

Thus condition 3 is satisfied, if we choose $n_{l+1} \geq N_0$. By condition 4, we know that for all $i < l+1$, we have

\[ \overline{\rho}( \xi^{m_i}, \xi^{m_{l+1}}) < \frac{1}{2iK}. \]

Therefore, we find some natural numbers $N_i$, such that

\[ \forall n \geq N_i, \rho( x^{(m_i)}_n, x^{(m_{l+1})}_n ) < \frac{1}{iK}. \]

We put $N := \max(N_0, N_1, \dots, N_l, n_l)$ and get

\[ \forall n \geq N, \forall i < l+1, \rho( x^{(m_i)}_n, x^{(m_{l+1})}_n) < \frac{1}{iK}. \]

Put $n_{l+1} := N$ and put

\[ x_{l+1} := x^{(m_{l+1})}_{n_{l+1}}. \]

By the definition of $N$, the sequence $(n_l)_l$ satisfies condition 2. Condition 3 is satisfied as $n_{l+1} \geq N_0$. Condition 4 is satisfied by choice of $m_{l+1}$ Finally, condition 5 is satisfied because $n_{l+1} \geq \max(N_1, \dots, N_l)$. Condition 1 is trivially satisfied and hence we have constructed a sequence with properties 1-5. We have seen before that such a sequence is a Cauchy sequence in $(Z, crt)$ and $(\xi^m)_m$ converges to $(x_l)_l$ in $(\overline{Z}, \overline{crt})$. Hence the Cauchy sequence $(\xi^m)_m$ converges. This implies that $(\overline{Z}, \overline{crt})$ is complete.

\end{proof}

\bibliographystyle{alpha}
\bibliography{mybib}

\begin{thebibliography}{Ham97}

\bibitem[BS07]{BS}
Sergei Buyalo and Viktor Schroeder.
\newblock {\em Elements of asymptotic geometry}.
\newblock EMS Monographs in Mathematics. European Mathematical Society (EMS),
  Z\"{u}rich, 2007.

\bibitem[BS17]{BeyrerSchroeder}
Jonas Beyrer and Victor Schroeder.
\newblock Trees and ultrametric {M}\"{o}bius structures.
\newblock {\em p-Adic Numbers Ultrametric Anal. Appl.}, 9(4):247--256, 2017.

\bibitem[Buy16]{Buyalo}
S.~V. Buyalo.
\newblock M\"obius and sub-m\"obius structures.
\newblock {\em Algebra i Analiz}, 28(5):1--20, 2016.

\bibitem[Ham97]{Hamenstadt97}
Ursula Hamenst\"{a}dt.
\newblock Cocycles, {H}ausdorff measures and cross ratios.
\newblock {\em Ergodic Theory Dynam. Systems}, 17(5):1061--1081, 1997.

\end{thebibliography}

\end{document}